\documentclass[reqno,11pt]{amsart}
\usepackage{amsmath}
\usepackage{mathrsfs}
\usepackage{bbm}

\usepackage{amsmath, amssymb, amsthm, graphicx}

\usepackage{color}
\usepackage{ifpdf}
\ifpdf 
  \usepackage[pdftex]{graphicx}
  \DeclareGraphicsExtensions{.pdf,.png,.jpg,.jpeg,.mps}
  \usepackage{pgf}
  \usepackage{tikz}
\else 
  \usepackage{graphicx}
  \DeclareGraphicsExtensions{.eps,.bmp}
  \usepackage{pgf}
  \usepackage{tikz}
  \usepackage{pstricks}
\fi
\usepackage{epic,bez123}
\usepackage{floatflt}
\usepackage{wrapfig}

\newtheorem{lemma}{Lemma}[section]
\newtheorem{theorem}{Theorem}[section]

\theoremstyle{definition}
\newtheorem{definition}{Definition}[section]
\theoremstyle{remark}
\newtheorem{remark}{Remark}[section]

\numberwithin{equation}{section}

\newcommand{\p}{\partial}
\newcommand{\bfn}{\mathbf{n}}
\newcommand{\bfu}{\mathbf{u}}
\newcommand{\bfr}{\mathbf{r}}
\newcommand{\norm}[1]{\left\Vert#1\right\Vert}
\newcommand{\dd}{\mathrm{d}}

\newcommand{\TV}{\mathrm{T.V.}}

\makeatletter
\newcommand{\rmnum}[1]{\romannumeral #1}
\newcommand{\Rmnum}[1]{\expandafter\@slowromancap\romannumeral#1@}
\makeatother

\begin{document}
\title[Stability of Transonic Characteristic Discontinuities]
{Stability of Transonic Characteristic Discontinuities in Two-Dimensional
Steady Compressible Euler Flows}

\author{Gui-Qiang Chen}
\author{Vaibhav Kukreja}
\author{Hairong Yuan}

\address{Gui-Qiang G. Chen, School of Mathematical Sciences, Fudan University,
 Shanghai 200433, China; Mathematical Institute, University of Oxford,
         Oxford, OX1 3LB, UK; Department of Mathematics, Northwestern University,
         Evanston, IL 60208, USA}
\email{\tt chengq@maths.ox.ac.uk}

\address{Vaibhav Kukreja, Department of Mathematics, Northwestern University,
         Evanston, IL 60208, USA}
         \email{\tt vkukreja@math.northwestern.edu}

\address{Hairong Yuan,
Department of Mathematics, East China Normal University, Shanghai
200241, China} \email{hryuan@math.ecnu.edu.cn;\
hairongyuan0110@gmail.com}

\date{\today}

\begin{abstract}
For a two-dimensional steady supersonic Euler flow past a
convex cornered wall with right angle,
a characteristic discontinuity (vortex sheet and/or entropy wave) is generated, which
separates the supersonic flow from the gas at rest (hence subsonic).
We proved that such a transonic characteristic discontinuity is structurally stable
under small perturbations of the
upstream supersonic flow in $BV$.
The existence of a weak entropy solution and Lipschitz continuous
free boundary (i.e. characteristic discontinuity) is established.
To achieve this, the problem is formulated as a free boundary
problem for a nonstrictly hyperbolic system of conservation laws; and
the free boundary problem is then solved by analyzing nonlinear wave interactions
and employing the front tracking method.
\end{abstract}

\maketitle


\section{Introduction and Main Theorem}

We are concerned with the structural stability of
transonic characteristic discontinuities in two-dimensional steady full
compressible Euler flows, which separate supersonic flows from the static
gases (that is, flows with zero-velocity, hence subsonic, cf. Figure
\ref{fig1}) under small perturbations in the space of functions of bounded
variation of the upstream supersonic flow.
The flow is governed by the two-dimensional full Euler system,
consisting of the conservation laws of mass, momentum, and energy:
\begin{eqnarray}\label{euler}
\begin{cases}\p_x(\rho u)+\p_y(\rho v)=0,\\
\p_x(\rho u^2+p)+\p_y(\rho uv)=0,\\
\p_x(\rho uv)+\p_y(\rho v^2+p)=0,\\
\p_x(\rho u(E+\frac{p}{\rho}))+\p_y(\rho v(E+\frac{p}{\rho}))=0.
\end{cases}
\end{eqnarray}
As usual, the unknowns $\bfu=(u,v), p$, and $\rho$ are respectively
the velocity, the pressure, and the density of the flow, and
$$
E=\frac{1}{2}(u^2+v^2)+e(p,\rho)
$$
is the total energy per unit mass with the internal energy $e(p,\rho)$.
Let $S$
be the entropy. For polytropic gas, the constitutive relations are
$$
p=\kappa \rho^\gamma\exp(\frac{S}{c_\nu}), \qquad
e=\frac{(\gamma-1)p}{\rho}
$$
for some positive constants
$\kappa, c_\nu$, and $\gamma>1$.
The sonic speed is given by
$$
c=\sqrt{\frac{\gamma p}{\rho}}.
$$
The flow is said to be {\it supersonic}
(resp. {\it subsonic}) at a state point if $u^2+v^2>c^2$ (resp.
$u^2+v^2<c^2$) there. It is well-known that the Euler
system \eqref{euler} is hyperbolic for supersonic flow, and
particularly hyperbolic in the positive $x$-direction if $u>c$;
while it is of hyperbolic-elliptic composite-mixed type if the flow
is subsonic. Hereafter, we use $U=(u,v,p,\rho)$ to represent
the state of the flow under consideration.

An important physical case in which a characteristic discontinuity is generated
is as follows: the characteristic discontinuity is a straight line emerging
from a corner $O$ (that is the positive $x$-axis); the gas flow above
(i.e., in $\{x\in\mathbb{R}, y>0\}$) is a uniform supersonic flow
with the velocity $(\underline{u},0)$, pressure
$\underline{p}$, and density $\underline{\rho}^+$ such that
 $\underline{u}>\underline{c}^+$ for the sonic speed $\underline{c}^+>0$;
below the
characteristic discontinuity (i.e., in $\{x>0, y<0\}$), the gas is at rest
with zero-velocity,  pressure $\underline{p}$, and density
$\underline{\rho}^-$.
The question is whether such a transonic characteristic
discontinuity is structurally stable under small
perturbations of the upstream supersonic flow in the framework of
two-dimensional steady full Euler equations, as shown in Figure \ref{fig1}.
Notice that the characteristic discontinuity is either a combination of
a vortex sheet and an entropy wave or one of them.

For related cases, when the flows on both sides of the characteristic
discontinuity are supersonic, it has been shown to be structurally
stable by Chen-Zhang-Zhu \cite{Chen-Zhang-Zhu-2006} in the
framework of weak entropy solutions, and the $L^1$--stability
also holds as established by Chen-Kukreja \cite{Chen-Kukreja};
when the flow is in an infinite duct and on
both sides of the characteristic discontinuity the flows are subsonic,
Bae \cite{Bae-2011} proved that it is stable under small perturbations of the
walls of the duct. Characteristic discontinuities appear ubiquitously
in Mach reflection and refraction/reflection of shock upon an
interface.
For such problems, Chen
\cite{Chen2006} and Chen-Fang \cite{Chen-Fang-2008} studied the
stability of subsonic characteristic discontinuities; Fang-Wang-Yuan
\cite{Fang-Wang-Yuan2011} showed the local stability of supersonic
characteristic discontinuity in the framework of classical solutions.
Also see Zhang \cite{ZhangYQ} for supersonic potential flows past
a convex cornered bending wall and related geometry.
As far as we know, there have been no results available so far
concerning transonic
characteristic discontinuities when the supersonic flows
are not $C^1$ but only
belong to the space of functions of bounded variation.

We remark that considerable progress has been made on the existence and stability
of multidimensional transonic shocks in steady full Euler flows (see, for example,
\cite{Chen-Chen-Feldman2007,Chen2006,Liu-Yuan2008,Xin-Yan-Yin2009,Yuan2006}; also cf. \cite{Da}).
In these papers, the smooth supersonic
flow is given, and the key point is to solve a one-phase elliptic free
boundary problem. However, in order to solve the perturbed
characteristic discontinuity in this paper, the key point is to solve a
hyperbolic free boundary problem in the framework of weak entropy
solutions.

\begin{figure}[h]\label{fig1}
\centering
  \setlength{\unitlength}{1bp}%
  \begin{picture}(300, 200)(40,50)
  \put(0,0){\includegraphics[scale=0.70]{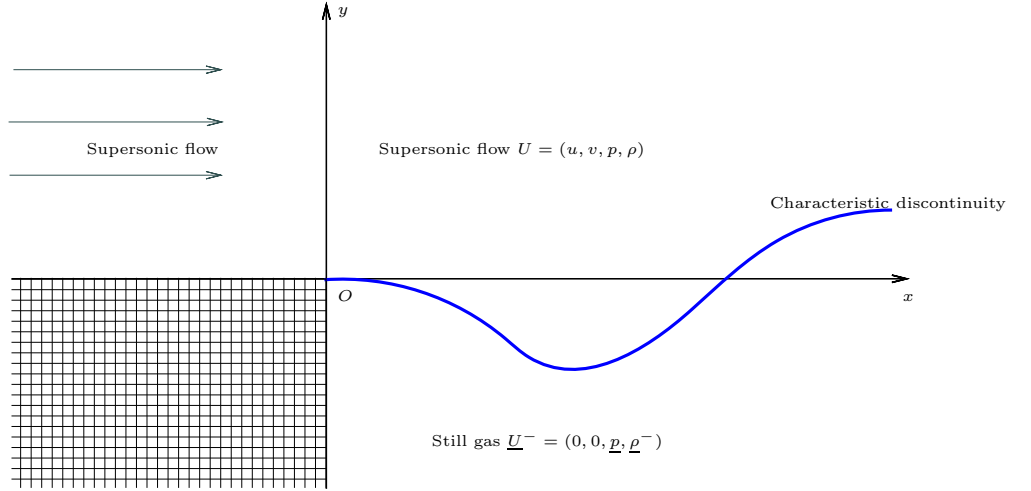}}

   \put(400,115){\fontsize{6}{5}\selectfont ${\tiny x}$}
   \put(187,223){\fontsize{6}{5}\selectfont ${\tiny y}$}
     \put(187,115){\fontsize{6}{5}\selectfont ${\tiny O}$}
\put(90,170){\fontsize{6}{5}\selectfont { Supersonic flow}}
\put(200,170){\fontsize{6}{5}\selectfont { Supersonic flow
$U=(u,v,p,\rho)$}} \put(220,60){\fontsize{6}{5}\selectfont { Still
gas $\underline{U}^-=(0,0,\underline{p},\underline{\rho}^-)$}}
\put(350,150){\fontsize{6}{5}\selectfont {Characteristic discontinuity}}
\end{picture}
\caption{\small   A characteristic discontinuity emerged from the corner
$O$ that separates the static gas with zero-velocity below from the supersonic
flow above. }
\end{figure}

In the following, we first formulate the aforementioned stability problem for
the characteristic discontinuity as a free boundary problem for the Euler
equations. Then, in Sections 2--5, we establish the existence and
stability of the free boundary, by a front
tracking method (cf. \cite{Bressan2000,Da,Holden-Risebro2002}).

To this end, we now introduce {\it characteristic discontinuities}, a kind of
discontinuities that separate piecewise classical/weak solutions of
\eqref{euler}. Suppose that $\Gamma$ is a Lipschitz curve with normal
$\bfn=(n_1,n_2)$ in the plane, and the flows $U=(u,v,p,\rho)$ on both
sides of $\Gamma$ satisfy the Euler equations \eqref{euler} in the
classical/weak sense. Then $U$ is a weak solution to \eqref{euler}
provided it satisfies \eqref{euler} on either side of $\Gamma$ in the
classical/weak sense, and the following Rankine-Hugoniot jump
conditions hold along $\Gamma$:
\begin{eqnarray}\label{rh}\begin{cases}
[\rho u]n_1+[\rho v]n_2 =0,\\
[\rho u^2+p] n_1+ [\rho uv]n_2 =0,\\
[\rho uv]n_1+[\rho v^2+p]n_2 =0,\\
[\rho u(E+\frac{p}{\rho})]n_1+[\rho v(E+\frac{p}{\rho})]n_2 =0,\end{cases}
\end{eqnarray}
where $[\cdot]$ denotes the jump of the quantity across $\Gamma$.
Such a
discontinuity $\Gamma$ is called a {\it characteristic discontinuity} if the
mass flux  $m=\rho \bfu\cdot\bfn=(\rho u)n_1+(\rho v)n_2$ through
$\Gamma$ is zero.
For a characteristic discontinuity, the first and fourth
condition ($[\rho\bfu\cdot\bfn (E+\frac{p}{\rho})]=0$) in \eqref{rh} hold
trivially, while the second ($[u\rho\bfu\cdot\bfn]+[p]n_1=0$) and
the third  ($[v\rho\bfu\cdot\bfn]+[p]n_2=0$) imply $[p]=0$.
Thus, we see that, for a characteristic discontinuity,
the only jump conditions
should be
\begin{eqnarray}\label{rh2}
[p]=0\qquad\text{and}\qquad \bfu\cdot\bfn=0.
\end{eqnarray}
This implies that there might be jumps of the tangential velocity and the
entropy (i.e., the density). Therefore, in general, a characteristic discontinuity
in full Euler flow is either a vortex sheet or an entropy wave.
We also note that \eqref{rh2} implies \eqref{rh}.

Consider the Cauchy problem of the
hyperbolic-elliptic composite-mixed system \eqref{euler}:
\begin{eqnarray}\label{prob1}
\begin{cases}
\eqref{euler} \qquad\text{in} \ \ x\ge0,\ y\in\mathbb{R},\\
U=\begin{cases}
U_0, & x=0,\ y>0,\\
\underline{U}^-, &x=0,\ y<0.
\end{cases}
\end{cases}
\end{eqnarray}
The discontinuous function:
$$
U=\begin{cases}\underline{U}^+=(\underline{u}, 0, \underline{p},
\underline{\rho}^+), &x>0,\ y>0,\\
\underline{U}^-=(0,0,\underline{p},\underline{\rho}^-), & x>0,\ y<0,
\end{cases}
$$
with
$\underline{u}>\underline{c}^+=\sqrt{\gamma\underline{p}/\underline{\rho}^+}$
is a characteristic discontinuity of  \eqref{euler},
when $U_0=\underline{U}^+$, and $\underline{U}^-$ is the state of
the static gas below $\{x>0,\ y=0\}$.

A weak entropy solution to problem \eqref{prob1} can be defined
in the standard way (cf. Definition \ref{def11} below):
In particular, it is defined as in \eqref{weak1}--\eqref{weak5}, but
the domain of integration $\Omega$ is replaced by $\{x\ge0,\
y\in\mathbb{R}\}$, $\Sigma$ is replaced by $\{x=0,\
y\in\mathbb{R}\}$, and the right-hand sides of \eqref{weak2}--\eqref{weak3}
are replaced by zero.

We note that the state of the static gas $\underline{U}^-$ should be
unchanged under the perturbation of the supersonic flow. This is a
merit of such a transonic characteristic discontinuity, which enables us to
reduce the above problem to an initial-free boundary problem of the
hyperbolic Euler equations.

Suppose that the characteristic discontinuity $\Gamma$ is given by the equation:
$$
y=g(x) \quad\mbox{for}\,\,\, x\ge0,
$$
with $g(0)=0$. Then
$$
\bfn=\frac{(g'(x),-1)}{\sqrt{1+(g'(x))^2}}.
$$
The domain bounded by $\Gamma$ and
$\Sigma=\{(x,y)\,:\, x=0, y>0\}$ is written as $\Omega$.
We formulate the following free boundary problem of
\eqref{euler} in $\Omega$:
\begin{eqnarray}\label{prob}
\begin{cases}
U=U_0 &\text{on}\ \ \Sigma,\\
p=\underline{p} &\text{on}\ \ \Gamma,\\
v=g'(x)u&\text{on}\ \ \Gamma,
\end{cases}
\end{eqnarray}
where the first is the initial data and the last two conditions on $\Gamma$
come from \eqref{rh2}.

\begin{definition}\label{def11}
A pair $(g,U)$ with $y=g(x)\in \rm{Lip}([0,\infty);\mathbb{R})$  and
$U=(u,v,p,\rho)\in L^\infty(\Omega;\mathbb{R}^4)$ is called a {\it
weak entropy solution} to problem \eqref{prob} provided the
following hold:
\begin{itemize}
\item[$\diamondsuit$] $U$ is a weak solution to \eqref{euler} in
$\Omega$ and satisfies the initial-boundary conditions in
the trace sense: For any $\phi\in C_0^\infty(\mathbb{R}^2)$,
\begin{eqnarray}
&&\int_{\Omega}\big(\rho u\p_x\phi+\rho v \p_y\phi\big)\,\dd x\dd
y+\int_\Sigma \rho u\phi\,\dd y=0, \label{weak1}\\
&&\int_{\Omega}\big((\rho u^2+p)\p_x\phi+\rho uv \p_y\phi\big)\,\dd x\dd
y+\int_\Sigma (\rho u^2+p)\phi\,\dd y
=\underline{p}\int_\Gamma\phi
n_1\,\dd s, \label{weak2}\\
&&\int_{\Omega}\big((\rho uv)\p_x\phi+(\rho v^2+p) \p_y\phi\big)\,\dd x\dd
y+\int_\Sigma (\rho uv)\phi\,\dd y=\underline{p}\int_\Gamma\phi n_2\,\dd
s, \label{weak3}\\
&&\int_{\Omega}\big(\rho u (E+\frac{p}{\rho})\p_x\phi+\rho v
(E+\frac{p}{\rho})\p_y\phi)\,\dd x\dd
y+\int_\Sigma \rho
u(E+\frac{p}{\rho})\phi\,\dd y=0;\label{weak4}
\end{eqnarray}

\item[$\diamondsuit$] $U$ satisfies the entropy inequality, i.e., the
steady Clausius inequality:
\begin{eqnarray*}
\p_x(\rho u S)+\p_y(\rho vS)\ge0
\end{eqnarray*}
in the sense of distribution in $\Omega$: For any $\phi\in
C_0^\infty(\mathbb{R}^2)$ with $\phi\ge0$:
\begin{eqnarray}\label{weak5}
\int_{\Omega}\big(\rho u S \p_x\phi+\rho v S\p_y\phi\big)\,\dd x\dd
y+\int_\Sigma \rho uS\phi\,\dd y\le0.
\end{eqnarray}
\end{itemize}
\end{definition}

We remark that, if $(g, U)$ is a weak entropy solution to problem
\eqref{prob}, then
$$\tilde{U}=\begin{cases}
U & \text{in}\ \ \{y>g(x),\ x\ge0\},\\
\underline{U}^- &\text{in}\ \  \{y<g(x), \ x\ge0\}
\end{cases}$$
 is a weak entropy solution to problem \eqref{prob1}.
This can be checked by integration by parts in
$\{x\ge0, y<g(x)\}$; thus, we omit the details.
From now on, we focus
on the solution of problem \eqref{prob}. The main result of this
paper is the following.

\begin{theorem}\label{thm1}
There exists positive constants $\varepsilon$ and $C$ depending only on
$\underline{U}^\pm$ so that, if
$$
\norm{U_0-\underline{U}^+}_{\rm{BV}(\Gamma)}\le\varepsilon,
$$
then
problem \eqref{prob} has a  weak entropy solution $(g,U)$.
Moreover, the
solution satisfies
\begin{itemize}
\item[(\rmnum{1})] $g\in \rm{Lip}([0,\infty);\mathbb{R})$  with
$g(0)=0$ and $\norm{g'}_{L^\infty[0,\infty)}\le C\varepsilon;$

\item[(\rmnum{2})] There exists  $\underline{U}_0\in\mathbb{R}^4$ so that
$$
U-\underline{U}_0\in C([0,\infty);L^1(g(x),\infty)), \qquad
\norm{(U-\underline{U}^+)(x,\cdot)}_{\rm{BV}([g(x),\infty))}\le
C\varepsilon.
$$
\end{itemize}
\end{theorem}

\begin{remark}\label{remark1}
We note that
$\norm{U_0-\underline{U}^+}_{\rm{BV}(\Sigma)}\le\varepsilon$ implies
that $\lim_{y\to\infty}(U_0-\underline{U}^+)(y)$ exists. Then
there exists $\underline{U}_0\in\mathbb{R}^4$ as claimed in Theorem
\ref{thm1} so that
\begin{eqnarray*}
\lim_{y\to\infty}U_0(y)=\underline{U}_0,
\end{eqnarray*}
and
\begin{equation*}
|\underline{U}_0-\underline{U}^+|\le\varepsilon.
\end{equation*}
\end{remark}

To prove Theorem \ref{thm1}, we establish the compactness and convergence of approximate free boundaries
to the free boundary of the exact solution in supersonic-subsonic flows in the framework of
front tracking method, while some other essential tools/notions of the front tracking method
are extended, modified, and further clarified working in the presence
of the free boundary such as a generation of fronts to control the finiteness of physical fronts and
the errors from approximate Riemann solvers for the nonstrictly hyperbolic free boundary problem.
For this, two new nonlinear Riemann problems are involved:
One is the Riemann problem at the convex corner connected with the still gas state (subsonic state);
and the other is the Riemann problem determining the evolution of the free boundary, for which we establish
the boundedness of the key reflection coefficient of the reflected wave into the supersonic region after
the interaction of the incident wave with the free boundary.
To achieve the compactness, we have to identify the right scales and global weights to control
the Glimm functional to make it monotonically decrease in the flow direction,
while preserving the overall structural stability of the characteristic boundary as
the hyperbolic region evolves in complicated ways under any small BV perturbation
yet the subsonic state remains stable beneath the free boundary.

We also remark in passing that, as an example of one-phase hyperbolic free boundary
problems for nonstrictly hyperbolic systems,
we deal with the problem
in the physical space, the Euler coordinates throughout this paper.
This represents a first example of an approach to apply the front-tracking method
to study structural stability of interfaces between different mediums, one of them is subsonic.
Our approach offers further opportunities to initiate the study of vortex sheets/entropy waves
in the space of bounded variation in nozzles, jets, etc. for mixed-type flows, transonic flows.
In a forthcoming paper, we will deal with
this problem and related $L^1$-stability in a different approach.

The rest of this paper is devoted to establishing Theorem \ref{thm1}. We
will mainly employ a version of the front tracking method introduced
in Holden-Risebro \cite{Holden-Risebro2002} for convenience to deal with the problem.
Thus, in Section 2, we review some facts
concerning the solvability of various Riemann problems for the steady Euler
equations, and present some essential interaction estimates. It
manifests clearly in the simplest case how such a hyperbolic free
boundary problem can be solved. Then, in Section 3, we construct
approximate solutions by the front tracking algorithm. The key point
is to show such an approximate solution can be established for
$x\in[0,\infty)$ by constructing a Glimm functional. Then, in Section
4, with the uniform $BV$ estimate of
approximate solutions obtained from the Glimm functional, we establish
the compactness of the family of approximate solutions and that the
limit is actually an entropy solution. Finally, we discuss the
asymptotic behavior of the weak entropy solutions as $x\to\infty$ in
Section 5.

\section{Riemann problems and interaction estimates}
\label{sec2}
In this section we first review certain basic properties of the steady
hyperbolic Euler equations \eqref{euler} that are used later for self-containedness
(cf. Chen-Zhang-Zhu \cite[pp.1665-1670]{Chen-Zhang-Zhu-2006}).
Then we show the solvability of ``free boundary"
Riemann problem and interaction estimate between weak waves and the
free boundary, which are the \textit{new} ingredients in this paper.

\subsection{Euler Equations}\label{sec21}
As in \cite{Chen-Zhang-Zhu-2006},
we write the Euler equations
\eqref{euler} in the form
\begin{eqnarray}\label{eq21}
\p_xW(U)+\p_yH(U)=0, \qquad U=(u,v,p,\rho),
\end{eqnarray}
where
$$
W(U)=(\rho u, \rho u^2+p,\rho uv,\rho
u(\frac{\gamma p}{(\gamma-1)\rho}+\frac{u^2+v^2}{2}))^\top
$$
and
$$
H(U)=(\rho v,\rho uv, \rho v^2+p, \rho v(\frac{\gamma
p}{(\gamma-1)\rho}+\frac{u^2+v^2}{2}))^\top.
$$

The eigenvalues $\lambda$
of this system are determined by
$\det(\lambda\nabla_UW(U)-\nabla_UH(U))=0$, or explicitly,
$$
(v-\lambda u)^2\big((v-\lambda u)^2-c^2(1+\lambda^2)\big)=0.
$$
Thus, if $u>c$, we have four real eigenvalues:
\begin{eqnarray}\label{2.2}
\lambda_j=\frac{uv+(-1)^j\sqrt{u^2+v^2-c^2}}{u^2-c^2},\quad
j=1,4;\qquad\;\; \lambda_k=\frac{v}{u},\quad k=2,3.
\end{eqnarray}
The associated linearly independent right-eigenvectors are
\begin{eqnarray}
&&\bfr_j=\kappa_j(-\lambda_j, 1,\rho(\lambda_ju-v),
\frac{\rho(\lambda_ju-v)}{c^2})^\top,\qquad j=1,4; \\
&&\bfr_2=(u,v,0,0)^\top,\qquad \bfr_3=(0,0,0,\rho)^\top,
\end{eqnarray}
where $\kappa_j$ are renormalized factors so that
$\bfr_j\cdot\nabla_U\lambda_j(U)\equiv1$ since
the $j$-th characteristic fields are genuinely nonlinear,
$j=1,4$. While the second and third characteristic fields are
linearly degenerate: $\bfr_j \cdot \nabla_U\lambda_j(U)\equiv0, j=2,3$.
Although the steady Euler system is not strictly hyperbolic, we
can still employ the general ideas presented in
\cite{Da,Holden-Risebro2002} to treat related Riemann and Cauchy
problems.
The only difference is that, although the characteristic discontinuity
has only one front in physical space (since two of the four characteristic eigenvalues coincide),
we need two independent parameters (one corresponds to $\lambda_2$ for the vortex sheet, and the
other to $\lambda_3$ for the entropy wave) to represent its
strength.

At the unperturbed reference state $\underline{U}^+=(\underline{u},
0, \underline{p},\underline{\rho}^+)$, we easily see that
\begin{eqnarray*}
\lambda_1(\underline{U}^+)<\lambda_2(\underline{U}^+)=0=\lambda_3(\underline{U}^+)
<\lambda_4(\underline{U}^+)=-\lambda_1(\underline{U}^+).
\end{eqnarray*}
Also, Lemma 2.3 in \cite{Chen-Zhang-Zhu-2006} indicates that the
re-normalization factors $\kappa_j(U), j=1,4,$ are positive in a
small neighborhood of $\underline{U}^+$.

\subsection{Wave Curves in the Phase Space}\label{sec22}
As shown in \cite{Chen-Zhang-Zhu-2006},
at each state
$U_0=(u_0,v_0,p_0,\rho_0)$ with $u_0>c_0$ in the phase space, there
are four curves in a neighborhood of $U_0$:
\begin{itemize}
\item[$\diamondsuit$]Vortex sheet curve $C_2(U_0): U=(u_0e^{\alpha_2}, v_0e^{\alpha_2}, p_0,\rho_0).$

These are the states $U$ that can be connected to $U_0$ by a vortex
sheet with slope $\frac{v_0}{u_0}$ and strength $\alpha_2\in\mathbb{R}$;

\item[$\diamondsuit$]Entropy wave curve $C_3(U_0): U=(u_0,v_0,p_0,\rho_0e^{\alpha_3}).$

These are the states $U$ that can be connected to $U_0$ by an entropy
wave with slope $\frac{v_0}{u_0}$ and strength $\alpha_3\in\mathbb{R}$.

\item[$\diamondsuit$]Rarefaction wave curve $R_j(U_0)$:
$$
\dd p=c^2\dd\rho, \dd u=-\lambda_j\dd v, \rho(\lambda_j u-v)\dd v=\dd
p\qquad \mbox{for $\rho<\rho_0, u>c, \ \ j=1,4.$}
$$

These are the states $U$ that can be connected to $U_0$ from the lower
by a rarefaction wave of the $j$-th family;

\item[$\diamondsuit$]Shock wave curve $S_j(U_0)$:
$$
[p]=\frac{c_0^2}{b}[\rho],  [u]=-s_j[v], \rho_0
(s_j u_0-v_0) [v]=[p]\qquad \mbox{for $\rho>\rho_0, u>c, \ \ j=1,4$.}
$$

These are the states $U$ that can be connected to $U_0$ from the lower
 by a shock wave of the $j$-th family, with the slope of the
discontinuity to be
$$
s_j=\frac{u_0v_0+(-1)^j\bar{c}\sqrt{u_0^2+v_0^2-\bar{c}^2}}{u_0^2-\bar{c}^2},
\qquad j=1,4,
$$
where $\bar{c}=\frac{\rho c_0^2}{\rho_0 b}$ and
$b=\frac{\gamma+1}{2}-\frac{\gamma-1}{2}\frac{\rho}{\rho_0}$.
\end{itemize}

One can also parameterize $R_j(U_0)$ and $S_j(U_0)$ ($j=1,4$) so
that there is a curve given by a $C^2$ map $\alpha_j\mapsto
\Phi_j(\alpha_j;U_0)$ in a neighborhood of $U_0$, with
$\alpha_j\ge0$ being the part of $R_j(U_0)$, and $\alpha_j<0$ the
part of $S_j(U_0)$, and
\begin{eqnarray}\label{wavecurve1}
\Phi_j(0;U_0)=U_0,\qquad
\p_{\alpha_j}\Phi_j(0;U_0)=\bfr_j(U_0).
\end{eqnarray}
We can also
write the curve $C_j(U_0)$ $(j=2,3)$ as
$\alpha_j\mapsto\Phi_j(\alpha_j;U_0)$ which is still $C^2$ so that
\eqref{wavecurve1} hold for $j=2,3$.
Since
$\{\bfr_j(U_0)\}_{j=1}^4$ are linearly independent, such curves
consist locally a (curved) coordinate system in a neighborhood of
$U_0$. This guarantees the solvability of the Riemann problems stated below.

For simplicity, we set
\begin{eqnarray}
\Phi(\alpha_4,\alpha_3,\alpha_2,\alpha_1;U_0)=\Phi_4(\alpha_4;\Phi_3(\alpha_3;\Phi_2(\alpha_2;\Phi_1(\alpha_1;U_0)))).
\end{eqnarray}
Then
\begin{eqnarray}
\Phi(0,0,0,0;U_0)=U_0,
\quad\p_{\alpha_j}\Phi(0,0,0,0;U_0)=\bfr_j(U_0), \qquad\; j=1,2,3,4.
\end{eqnarray}

\subsection{Standard Riemann Problem}\label{sec23}
We now consider the standard Riemann problem, that is, system
\eqref{euler} with the piecewise constant (supersonic) initial data
\begin{eqnarray}\label{RP}
U|_{x=x_0}=\begin{cases}
U^+,& \quad y>y_0,\\
U^-,& \quad y<y_0,
\end{cases}
\end{eqnarray}
where $U^+$ and $U^-$ are the constant states which are regarded as the \textit{above}
state and \textit{below} state with respect to the line $y=y_0$, respectively.

\begin{lemma}[Lemma 2.2 in \cite{Chen-Zhang-Zhu-2006}]\label{lem:riemansolver}
There exists $\epsilon>0$ such that, for any states $U^-$ and $U^+$ lie in
the ball $O_\epsilon(U_0)\subset\mathbb{R}^4$ with radius $\epsilon$
and center $U_0$, the above Riemann problem admits a unique
admissible solution consisting of four elementary waves. In
addition, the state $U^+$ can be represented by
\begin{eqnarray}\label{riemsolver}
U^+=\Phi(\alpha_4,\alpha_3,\alpha_2,\alpha_1;U^-).\end{eqnarray}
\end{lemma}

It is noted (cf. Lemma 4.1 in \cite{Chen-Zhang-Zhu-2006})
that one can use the parameters $\alpha_j, j=1, \ldots, 4$, to bound $|U^+-U^-|$:
There is a constant $B$ depending continuously on $U_0$ and
$\epsilon$ so that, for $U^\pm$ connected by \eqref{riemsolver},
$$
\frac{1}{B}\sum_{j=1}^4|\alpha_j|\le|U^+-U^-|\le B\sum_{j=1}^4|\alpha_j|.
$$

For later applications, it is also important to express the Riemann
solver from the upper state $U^+$ to the lower state $U^-$,
rather than the usual way given above. For
$U^+=\Phi_j(\alpha_j;U^-)$, we may have a $C^2$--map
$U^-=\Psi_j(\alpha_j;U^+)$ with $\Psi_j(0;U)=U$ and
$\p_{\alpha_j}\Psi_j(0;U)=-\bfr_j(U)$.
Thus, for
$U^+=\Phi(\alpha_4,\alpha_3,\alpha_2,\alpha_1;U^-)$,
we may express
$U^-$ in terms of $U^+$ by
\begin{eqnarray*}
U^-=\Psi(\alpha_1,\alpha_2,\alpha_3,\alpha_4;U^+)
=\Psi_1(\alpha_1;\Psi_2(\alpha_2;\Psi_3(\alpha_3;\Psi_4(\alpha_4;U^+)))).
\end{eqnarray*}
Then $\Psi(0,0,0,0;U)=U$ and $\p_{\alpha_j}\Psi(0,0,0,0;U)=-\bfr_j(U)$.

\subsection{Free Boundary Riemann Problem}\label{sec24}

We now consider the following Riemann problem of \eqref{euler}
involving a free boundary---a characteristic discontinuity. The initial
data is a constant state $U=U^+$ given on the positive $y$-axis, and
the free boundary is a straight line $y=k x$ with $k\in\mathbb{R}$
to be solved. The  boundary conditions on the free boundary are
$p=\underline{p}$ and $k=\frac{v}{u}$. Since the free boundary -- characteristic
discontinuity --- is of the second/third  characteristic family, the
Riemann solver should contain only one 4-wave with parameter
$\alpha_4$ and a middle constant state $U^\star$; see Figure 2
below.
\begin{figure}[h]\label{fig2}
\centering
  \setlength{\unitlength}{1bp}%
  \begin{picture}(300, 200)(0,0)
  \put(0,0){\includegraphics[scale=0.70]{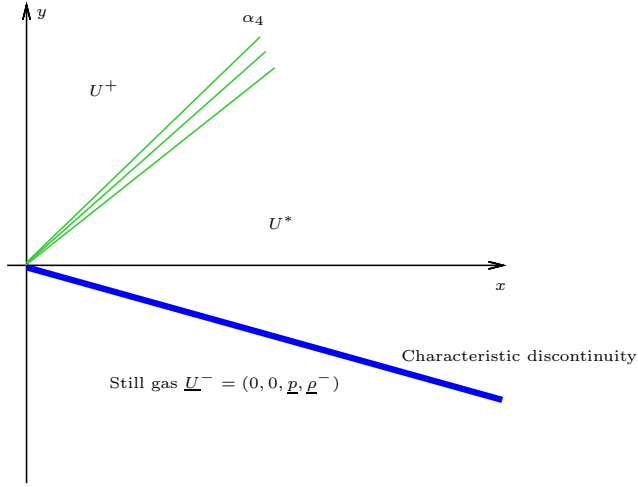}}
   \put(188,77){\fontsize{6}{5}\selectfont ${\tiny x}$}
   \put(15,180){\fontsize{6}{5}\selectfont ${\tiny y}$}
 \put(35,150){\fontsize{6}{5}\selectfont ${\tiny U^+}$}
\put(90,178){\fontsize{6}{5}\selectfont { $\alpha_4$}}
\put(100,100){\fontsize{6}{5}\selectfont { $U^\ast$}}
 \put(40,40){\fontsize{6}{5}\selectfont { Still
gas $\underline{U}^-=(0,0,\underline{p},\underline{\rho}^-)$}}
\put(150,50){\fontsize{6}{5}\selectfont { Characteristic discontinuity}}
  \end{picture}
\caption{\small   A Riemann problem with a free boundary that is a
characteristic discontinuity. }
\end{figure}

\begin{lemma}\label{lem:bounriesolver}
There exists $\epsilon>0$ so that, for $U^+\in
O_\epsilon(\underline{U}^+)$, there is only one admissible solution
consisting of a 4-wave that solves the above free boundary Riemann
problem. The middle state $U^\ast$ can be represented by
$U^\ast=\Psi_4(\alpha_4;U^+)$, and the free boundary is determined by
$k=\frac{v^\ast}{u^\ast}$.
There also holds
\begin{eqnarray}
\alpha_4= K_1(p^+-\underline{p})+M_1|U^+-\underline{U}^+|^2,\qquad
|k|\le K_1'|U^+-\underline{U}^+|,
\end{eqnarray}
with the constants $K_1, K_1'>0$ and a bounded quantity $M_1$ only
depending continuously on $\underline{U}^+$ and $\epsilon$.
\end{lemma}

\begin{proof}
1. We write $U^{(k)}$ to denote the $k$-th argument of the vector
$U$, $k=1,\ldots,4$. Consider the function:
\begin{eqnarray*}
L(\alpha,U^+)=(\Psi_4(\alpha;U^+))^{(3)}-\underline{p}
=(\Psi_4(\alpha;U^+)-\Psi_4(0;\underline{U}^+))^{(3)},
\end{eqnarray*}
for which $L(0;\underline{U}^+)=0$. Then
\begin{eqnarray*}
\p_\alpha
L(0;\underline{U}^+)=-(\bfr_4(\underline{U}^+))^{(3)}=-(\kappa_4\rho
u\lambda_4)|_{\underline{U}^+}<0.
\end{eqnarray*}
From the implicit function theorem, we infer that $\alpha$ can be
viewed as a function of $U^+\in O_\epsilon(\underline{U}^+)$ for
suitably small $\epsilon>0$. In particular, $\alpha(\underline{U}^+)=0$.
This completes the existence proof.

\medskip
2. Since $\nabla_U\Psi_4(0;U)=I_4$,
$\p_U L(0;\underline{U}^+)=(0,0,1,0)$. Then
$$
\nabla_U\alpha(\underline{U}^+)=\frac{(0,0,1,0)}{(\kappa_4\rho
u\lambda_4)|_{\underline{U}^+}}.
$$
Thus, by the Taylor expansion, we conclude
$$
\alpha=K_1(p^+-\underline{p})+M_1|U^+-\underline{U}^+|^2,
$$
where $K_1=\frac{1}{(\kappa_4\rho u\lambda_4)|_{\underline{U}^+}}>0$,
and $M_1$ is a constant depending continuously and only on
$\underline{U}^+$ and $\epsilon$.

\medskip
3. From the above, we have
$$
|U^\ast-U^+|\le B|\alpha|\le
B'|U^+-\underline{U}^+|.
$$
Then we have
$$
|U^\ast-\underline{U}^+|\le
B''|U^+-\underline{U}^+|
$$
for some constant $B''>0$. Hence,
regarding $\frac{v}{u}$ as a function of $U$ and
by the mean value theorem,
we have
$$
|\frac{v^\ast}{u^\ast}|\le C|U^*-\underline{U}^+|\le
K_1'|U^+-\underline{U}^+|
$$
as desired.
\end{proof}

\subsection{Approximate Riemann Solver}\label{sec25}
The front tracking method involves approximating the rarefaction
waves appeared in the Riemann problems or (free) boundary Riemann
problems by several artificial discontinuities separating piecewise
constant states.

Suppose that $U^+=\Phi(\alpha_4,\alpha_3,\alpha_2,\alpha_1;U^-)$ gives
the solution to the standard Riemann problem \eqref{RP}, with middle
states $U^1=\Phi_1(\alpha_1;U^-)$ and $U^2=\Psi_4(\alpha_4; U^+)$.
For any $\delta>0$, we define a $\delta$-approximate solution
$U^\delta$ to the Riemann problem as follows:

\begin{itemize}
\item If $\alpha_1>0$, then the 1-wave is a rarefaction wave that requires
modification as follows. Set $\nu$ be the closest integer to
$\frac{\alpha_1}{\delta}$ (that is, $\nu\in\mathbb{Z}$ and
$\frac{\alpha_1}{\delta}-\frac{1}{2}\le\nu<\frac{\alpha_1}{\delta}+\frac{1}{2}$), as well as
$U_{1,0}=U^-,$ $U_{1,\nu}=U^1$, and
$U_{1,k}=\Phi_1(\frac{1}{\nu}\alpha_1;U_{1,k-1})$ for
$k\in\{1,\ldots,\nu-1\}$. Then,
in the wedge $\{(x,y): x>0,
y<\lambda_*x\}$, we define

\begin{eqnarray}\label{eqappr}
U^\delta=\begin{cases} U^-,& y<\lambda_1(U^-)x,\\
U_{1,k}, & \lambda_1(U_{1,k-1})x<y<\lambda_1(U_{1,k})x, \quad
k=1,\ldots,\nu-1,\\
U^1, &\lambda_1(U_{1,\nu-1})x<y<\lambda_*x.
\end{cases}
\end{eqnarray}
Here $\lambda_*$ is a constant chosen so that $\sup_{U\in
O_\epsilon(\underline{U}^+)}{\lambda_1}<\lambda_*<\inf_{U\in
O_\epsilon(\underline{U}^+)}{\lambda_2}$, which exists when $\epsilon$
is small.

Then the rarefaction wave is replaced by ``step" functions with width
(strength) $\frac{\alpha_1}{\nu}$, and the discontinuity between two steps
moves with the characteristic speed of the lower state.

\item If $\alpha_1<0$, then the 1-wave is a shock, and no change is necessary. In the wedge $\{(x,y): x>0,
y<\lambda_*x\}$, we define
$$U^\delta=\begin{cases}
U^-, &y<s_1x,\\
U^1, &s_1x<y<\lambda_*x,
\end{cases}$$
where $s_1$ is the speed of the shock front.

\item For $\alpha_2, \alpha_3$, there is always no change.

\item Similar to the case of the 1-wave, we can define $U^\delta$ in $\{x>0,y>-\lambda_*x\}$ by
considering whether the $4$-waves is a rarefaction wave (with
modification) or a shock (without modification).
\end{itemize}

\subsection{Interaction of Weak Waves}\label{sec26}
The following weak wave interaction estimate is classical; see Lemma
3.2 in \cite[p.1670]{Chen-Zhang-Zhu-2006}.

\begin{lemma}\label{lem23}
Suppose that $U^+,U^m$, and $U^-$ are three states in a small
neighborhood of $U_0$ with
$U^+=\Phi(\alpha_4,\alpha_3,\alpha_2,\alpha_1;U^m)$,
$U^m=\Phi(\beta_4,\beta_3,\beta_2,\beta_1;U^-)$, and
$U^+=\Phi(\gamma_4,\gamma_3,\gamma_2,\gamma_1;U^-)$. Then
\begin{eqnarray}
\gamma_j=\alpha_j+\beta_j+O(1)\triangle(\alpha,\beta),
\end{eqnarray}
where
$\triangle(\alpha,\beta)=|\alpha_4|(|\beta_1|+|\beta_2|+|\beta_3|)+(|\alpha_2|+|\alpha_3| )|\beta_1|
+\sum_{j=1,4}\triangle_j(\alpha,\beta)$, with
$$
\triangle_j(\alpha,\beta)=\begin{cases}
0, & \alpha_j\ge0,\ \beta_j\ge0,\\
|\alpha_j||\beta_j|, &\text{otherwise}.
\end{cases}$$
\end{lemma}

\subsection{Interaction of Weak Wave and Free Boundary}\label{sec27}
We now consider the change of strength when a weak
wave interacts with the free boundary (see Figure 3).
It
is only possible that a weak 1-wave $\alpha_1$ impinges on the
characteristic discontinuity $S^l$, and resulting a reflected 4-wave with
parameter $\alpha_4$, and the characteristic discontinuity itself is also
deflected to a new direction, denoted to be $S^r$. We note that both
$U^r$ and $S^r$ can be solved by the free boundary Riemann problem with
initial data $U^m$.

\begin{figure}[h]
\centering
  \setlength{\unitlength}{1bp}%
  \begin{picture}(300, 200)(0,-20)
  \put(0,0){\includegraphics[scale=0.70]{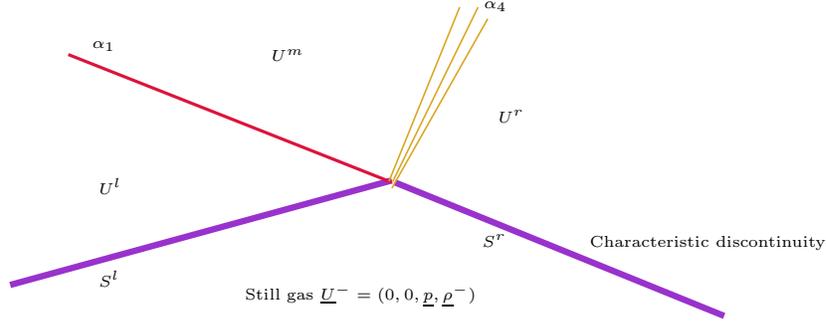}}

   \put(188,77){\fontsize{6}{5}\selectfont ${\tiny U^r}$}
    \put(182,30){\fontsize{6}{5}\selectfont ${\tiny S^r}$}
 \put(35,105){\fontsize{6}{5}\selectfont ${\tiny \alpha_1}$}
\put(180,120){\fontsize{6}{5}\selectfont { $\alpha_4$}}
\put(100,100){\fontsize{6}{5}\selectfont { $U^m$}}
\put(35,50){\fontsize{6}{5}\selectfont { $U^l$}}
\put(35,15){\fontsize{6}{5}\selectfont { $S^l$}}
 \put(90,10){\fontsize{6}{5}\selectfont { Still
gas $\underline{U}^-=(0,0,\underline{p},\underline{\rho}^-)$}}
\put(220,30){\fontsize{6}{5}\selectfont { Characteristic discontinuity}}
  \end{picture}\label{fig3a}
\caption{\small   A 1-wave $\alpha_1$ is reflected by the characteristic discontinuity
$S^l$, resulting in a reflected 4-wave $\alpha_4$ and deflected characteristic
discontinuity $S^r$. }
\end{figure}

\begin{lemma}\label{lem:reflection}
Suppose that $U^l,U^m$, and $U^r$ are three states in
$O_\epsilon(\underline{U}^+)$ for sufficiently small $\epsilon$,
with $U^m=\Phi_1(\alpha_1;U^l)=\Phi_4(\alpha_4;U^r)$. Then
\begin{eqnarray}\label{reflection}
\alpha_4=-K_2\alpha_1+M_2|\alpha_1|^2,
\end{eqnarray}
with the constant $K_2>0$ and the quantity $M_2$ bounded in
$O_\epsilon(\underline{U}^+)$.
Furthermore, for $U^l = (u_l, v_l, p_l, \rho_l)$,
$|K_2| > 1, |K_2| < 1, \text{ and } |K_2| =1 $ when $v_l<0$, $v_l>0$, \text{ and } $v_l=0$,
respectively.
\end{lemma}

\begin{proof}
1. We have $U^m=\Phi_1(\alpha;U^l)$ and $U^r=\Psi_4(\beta;U^m)$.
Consider the following  function:
$$
L(\beta,\alpha):=(\Psi_4(\beta;\Phi_1(\alpha;U^l))-U^l)^{(3)}.
$$
Then $L(0,0)=0,$ and $\p_\beta L(0,0)=-(\bfr_4(U^l))^{(3)}<0$.
By the implicit function theorem, there exists a function
$\beta=\beta(\alpha)$ so that $L(\beta(\alpha),\alpha)=0$ for
small $\alpha$. We see $\beta(0)=0.$

\medskip
2. We calculate $\p_\alpha L(0,0)=(\bfr_1(U^l))^{(3)}<0$. Thus, $\frac{\dd\beta(0)}{\dd\alpha}=-K_2:=\frac{(\bfr_1(U^l))^{(3)}}{(\bfr_4(U^l))^{(3)}}<0$.
Therefore the equality in \eqref{reflection} follows from Taylor
expansion.

3. The coefficient
$$
K_2:=-\frac{(\bfr_1(U^l))^{(3)}}{(\bfr_4(U^l))^{(3)}}
 = \frac{\frac{v_l}{u_l}-\lambda_{1}(U^l) }{\frac{v_l}{u_l}+\lambda_{4}(U^l)}>0
$$
and, for any state $U = (u,v,p,\rho) \in O_{\epsilon}(\underline{U}^+)$,
there holds $\lambda_{1}(U) < \lambda_{2,3}(U)= \frac{v}{u} < \lambda_{4}(U)$.
Using these two facts with the expressions for $\lambda_1(U)$ and $\lambda_4(U)$
given in \eqref{2.2}, it follows that $|K_2| < 1, |K_2| > 1,
\text{ and } |K_2| = 1$ when $v_l>0, v_l<0, \text{ and } v_l=0$, respectively.
\end{proof}

\section{Construction of Approximate Solutions and Uniform Estimates}
In this section we adopt the front tracking method in Holden-Risebro
\cite{Holden-Risebro2002} to construct a family of approximate
solutions $\{(g^\delta, U^\delta)\}_{\delta>0}$ of the problem
\eqref{prob} and present some uniform estimates independent of
$\delta$, which is necessary for a compactness argument in \S 4
to show the existence of a weak entropy solution to
\eqref{prob}.

\subsection{Construction of Approximate Solutions}

For any given $\delta>0$, we now describe the construction of an
approximate solution $(g^\delta, U^\delta)$ to the free boundary
problem \eqref{prob}.

We first approximate the initial data $U_0(y)$ by
a piecewise constant function $U_0^\delta(y)$ as done in the study
of the Cauchy problem. We require that
\begin{eqnarray}
\lim_{\delta\to 0}\big\|U_0-U_0^\delta\big\|_{L^1([0,\infty))}=0.
\end{eqnarray}
By Remark \ref{remark1}, we may also assume that, for each
$\delta>0$, there holds
$U_0^\delta(y)=\underline{U}_0$ for large $y$.

We solve the Riemann problems with initial data on $\{x=0, y>0\}$
and a free boundary Riemann problem at the corner $(0,0)$, and then
approximate rarefaction waves as carried out in \S 2.5 with
parameter $\delta$ to obtain new discontinuities. Note the
resulting (approximate) solution is piecewise constant.

Then we need do nothing until as $x$ increases to some value
$x=\tau$, where
\begin{itemize}
\item[(\rmnum{1})] either two fronts interact;

\item[(\rmnum{2})] or there is a weak 1-wave that interacts the free boundary
(it is obtained by solving the free boundary Riemann problem before)
from above.
\end{itemize}

As noted in \cite{Bressan2000}, by adjusting the slopes of the
discontinuities, we can assume that, at each $\{x=\tau\}$, only one of the
two cases above happens. This is harmless since the error can be
made to be arbitrarily small.

For case (\rmnum{1}), as mentioned above, by adjusting the slopes of
these discontinuities (with arbitrarily small error),
we may assume that only two discontinuities collide.
Suppose that the lower discontinuity is of $r$-family and has a
parameter $\alpha$ with the lower (constant) state $U^l$ and
upper (constant) state $U^m$,
the upper discontinuity is of $s$-family and has a
parameter $\beta$ with the lower (constant) state $U^m$ and
upper (constant) state $U^r$, and they collide at the point
$(\tau,\eta)$.
Then, as before, we solve a Riemann problem at $(\tau,\eta)$ with
the lower state $U^l$ and upper state $U^r$, by applying the
approximate Riemann solver to obtain new discontinuities.

For case (\rmnum{2}), we may still assume only one discontinuity
collides with the free boundary (transonic characteristic discontinuity).
Then we solve a wave reflection-deflection problem with a 1-wave
reflected by the free boundary, obtaining a reflected 4-wave and a deflected
characteristic discontinuity (see Figure 2).
If the reflected 4-wave is a rarefaction wave, by approximating the
rarefaction wave, we obtain again
the approximate solver containing new discontinuities.

Continuing this procedure and, in some cases, removing certain quite
weak fronts (cf. \S \ref{sec3.4.2} below for details), we
obtain an approximate solution $(g^\delta,U^\delta)$.

\begin{remark}
To ensure that the above procedure works to construct an approximate
solution for all $x\in[0,\infty)$, we need to show that, for any
$0<x<\infty$,
\begin{itemize}  \item The total variation is small:
$\mathrm{T.V.}(U^\delta(x,\cdot))\le C\varepsilon$;

\item An $L^\infty$--bound: The solution still lies in a small
neighborhood of $\underline{U}^+$;

\item Given any finite $T>0$,
there happens only a finite number of collisions/reflections
for $\{0<x<T\}$.
\end{itemize}
The first two are necessary so that we can actually solve the
standard or free boundary Riemann problem. Here $C$ is a universal
constant independent of $\varepsilon$ and $x>0$. The third one
guarantees that the global approximate solutions defined
up to any $x>0$ can be actually obtained.

In the following three subsections, we deal with these three issues.
\end{remark}

\subsection{Bounds of Total Variation}

We now establish the bounds of total variation of the approximate solutions $U^\delta(x,y)$.

\subsubsection{Glimm Functional}

We introduce the following version of Glimm functional
\begin{eqnarray}
G(x)=V(x)+\kappa Q(x),
\end{eqnarray}
where $\kappa > 0$ is a large constant to be chosen. The terms $V$
and $Q$ are explained below. By the properties of the approximate Riemann
solver, $\mathrm{T.V.}(U^\delta(x,\cdot))$ is equivalent to $V(x)$.
Then it suffices to prove
\begin{eqnarray}
V(x)\le C_0\varepsilon
\end{eqnarray}
for a constant $C_0$ depending only on $\underline{U}^+$. Recall
here $\varepsilon=\norm{U_0-\underline{U}^+}_{\rm{BV}([0,\infty))}$
measures the strength of the perturbation of initial data.

For a weak wave/discontinuity $\alpha$ of $i_\alpha$-family, we
define its weighted strength as
\begin{equation}
b_{\alpha} =
\begin{cases} k_{+}\alpha & \text{if $\alpha \in \Upsilon_t$ and $i_{\alpha} =1$},
\\
\alpha &\text{if $\alpha \in \Upsilon_t$ and $i_{\alpha} =2,3,4$},
\end{cases}
\end{equation}
where $k_{+} > |K_2|$ for the coefficient $K_{2}$ appeared in Lemma
\ref{lem:reflection}, and we use $\Upsilon_t$ to denote the set of
weak waves/discontinuities (not including the free boundary) that
cross the line $\{x=t\}$.

{$\bullet$\it The weighted strength term $V(t)$.} We define the
total (weighted) strengths of weak waves/discontinuities at $x=t$ as
\begin{eqnarray}
V(t) = \sum_{\alpha\in\Upsilon_t} |b_{\alpha}|.
\end{eqnarray}

{$\bullet$ \it The interaction potential term $Q(t)$.} The
interaction potential term we use here is the same one as
introduced by Glimm \cite{Glimm}, that is:
\begin{eqnarray}
Q(t) &=& \sum_{(b_\alpha, b_\beta) \in \mathcal{A}(t)} |b_{\alpha}
b_{\beta}|,
\end{eqnarray}
where $\mathcal{A}(t)$ is the {\it approaching set} defined by pairs
$(b_{\alpha},b_{\beta})$ so that, for $x=t$, the waves/discontinuity
with strength $b_{\alpha}$ lies in the lower side of the
waves/discontinuity with strength $b_{\beta}$, and $b_{\alpha}$ is
of family $i_{\alpha}$ and $b_{\beta}$ is of family
$i_{\beta}$, where $i_{\alpha} > i _{\beta}$, or both are of the same
family but at least one of them is a shock. Note we do not consider
the free boundary as a wave/discontinuity in this paper.

As shown by Lemma 6.2 in \cite{Holden-Risebro2002},
at $x=\tau$, if two discontinuities of strengths $b_{\alpha}$ and
$b_{\beta}$ collide, then we have
\begin{eqnarray}
Q(\tau+)-Q(\tau-)=-\frac{1}{2}|b_{\alpha}b_{\beta}|,
\end{eqnarray}
provided that
\begin{eqnarray}\label{TM1}
V(\tau-)\le \mu:=\frac{1}{2} O(1).
\end{eqnarray}
It is here one needs Lemma \ref{lem23}. If no discontinuities
collide at $x=\tau$, then $Q(\tau+)=Q(\tau-)$.

\subsubsection{Non-increasing of the Glimm Functional}
We now show the bounds of total variation by proving that the Glimm
functional $G(x)$ is non-increasing for $x$. There are the following
three cases.
\begin{itemize}
\item[(\rmnum{1})] {\it  Collision of discontinuities.} For $x=\tau$
where two discontinuities $b_{\alpha}$ and $b_{\beta}$ collide,
there is no other wave interaction and reflection upon the free
boundary as we assumed. Therefore, the decreasing of $G(\tau)$ is
classical. By Lemma \ref{lem23}, we have
\begin{eqnarray*}
G(\tau+)-G(\tau-)&=&(V(\tau+)-V(\tau-))
+\kappa(Q(\tau+)-Q(\tau-))\\
&\le& M |b_{\alpha}b_{\beta}|+\kappa(-\frac{1}{2}|b_{\alpha}b_{\beta}|
)\le 0,
\end{eqnarray*}
if we choose $\kappa\ge 2M$ sufficiently large. Note that $O(1)$ does
not depend on the approximation parameter $\delta$.

\item[(\rmnum{2})] {\it  Weak 1-wave interacts with the free boundary.} For $x=\tau$, a weak wave
$\alpha_1$ of 1-family interacts with the free boundary from above,
resulting in a reflected 4-wave $\alpha_4$. By Lemma
\ref{lem:reflection}, we have
\begin{eqnarray*}
G(\tau+)-G(\tau-)&=&(V(\tau+)-V(\tau-))
+\kappa(Q(\tau+)-Q(\tau-))\\
&\le&|b_{\alpha_4}|-|b_{\alpha_1}|+\kappa \mu|b_{\alpha_4}|\\
&\le&((\kappa \mu+1)(-K_2+M_2\mu)-k_+)|\alpha_1|\le 0
\end{eqnarray*}
if we choose $k_+$ sufficiently large (independent of $\delta$).

\item[(\rmnum{3})] {\it  Other situation.}  If, for $x=\tau$, no collision
or reflection upon the free boundary happens, then we still have
$G(\tau+)=G(\tau-)$.
\end{itemize}

In the above, we have determined $\kappa$ and $k_+$ independent of
$\delta$, and proved that, for any $x=\tau>0$, there holds $G(\tau+)\le
G(\tau-)$, provided \eqref{TM1} holds.

\subsubsection{Boundedness of Total Variation}
The bound $V(\tau)\le C_0\varepsilon$ then follows from an
induction argument as shown in \cite[p.217]{Holden-Risebro2002} for
the proof of Lemma 6.3 there, provided that $\varepsilon$ is small.

We first set $0<\tau_1<\ldots<\tau_k<\ldots$ as the sequence so
that, for $x=\tau_k$, either collision or reflection upon the free
boundary occurs, and set $V_k,\ G_k$ the value of $V(\tau_k-),\
G(\tau_k-)$ respectively.

We know that there exists a constant $C_1$ independent of $\delta>0$ so that
$V(\tau)\le C_1\rm{T.V.}(U^\delta(\tau,\cdot))$ for all $x\ge0$.
Note here the choice of weight $k_+$ is in essence only determined
by $\underline{U}^+$. Define
$$
C_0=C_1+\kappa C_1^2.
$$
We choose positive $\varepsilon<1$ small so that
$$
C_1\varepsilon+\kappa (C_1\varepsilon)^2\le\mu, \qquad C_2C_0\varepsilon
\le\epsilon.
$$
Here $\epsilon$ is the value so that the Riemann
problems or the free boundary Riemann problems can be solved when
the Riemann data are in $O_\epsilon(\underline{U}^+)$, and $C_2$ is the
constant depending only on $\underline{U}^+$ so that
${\rm{T.V.}}(U^\delta(x,\cdot))\le C_2 V(x)$ for any $x>0$.

By assumption on the initial data, we have $\TV(U_0^\delta)\le
\varepsilon$. Thus, by a property of the Riemann problem, we may have
$$
V_1\le C_1\varepsilon\le \min\{C_0\varepsilon, \mu\},
$$
and
furthermore,
$$
G_1\le V_1+\kappa V_1^2\le C_1\varepsilon+\kappa
(C_1\varepsilon)^2\le\min\{C_0\varepsilon,\mu\}.
$$

Suppose that, for $n\le k$, we have proved
$$
V_n\le
\min\{C_0\varepsilon,\mu\}.
$$
Then, by decreasing of the Glimm functional,
we have proved that there holds
$$
V_{k+1}\le G_{k+1}\le G_k\le\ldots \le G_1.
$$
This shows
$$
V_n\le\min\{C_0\varepsilon,\mu\}\qquad \mbox{for all $n$}.
$$
If we
further choose $\varepsilon$ small so that $C_0\varepsilon\le\mu$,
we obtain the bound $V(\tau)\le C_0\varepsilon$ as desired. This again
implies the uniform estimate:
\begin{eqnarray}\label{tvestimate}
{\rm T.V.}(U^\delta(x,\cdot))\le C_2C_0\varepsilon.
\end{eqnarray}

\subsection{$L^\infty$--Estimate of $\{U^\delta\}$ and
Lipschitz Estimate of $\{g^\delta\}$} The fact that
$\{U^{\delta}\}_{\delta>0}$ is uniformly bounded follows directly.
For each $x$, the solution $U^{\delta}(x,y)$ is just
the constant state $\underline{U}_0$ for sufficiently large $y$, by
the finiteness of propagation speed and the fact that
the initial data $U^\delta_0(y)\to \underline{U}_0$ for $y\to\infty$.
Since we have proved ${\rm T.V.}(U^{\delta}(t,\cdot)) \le
C_2C_0\varepsilon$ for any $t>0$, then, by definition of the total
variation, we conclude
\begin{eqnarray}\label{linfty}
\big\|U^{\delta}(x,\cdot)-\underline{U}^+\big\|_{L^{\infty}} \le
C_2C_0\varepsilon
\end{eqnarray}
for some new constant $C_2$.

Estimate \eqref{linfty} implies the following uniform
estimate on the free boundary that is given by the equation
$y=g^\delta(x)$:
\begin{eqnarray}\label{gbound}
\big\|(g^\delta)'\big\|_{L^\infty}\le C_3\varepsilon.
\end{eqnarray}
with a constant $C_3$ depending only on $\underline{U}^+$.
In particular, by construction, for fixed $\delta>0$, $g^\delta$ is a
piecewise linear (affine) function, and except for countable points
$\{\tau_k\}$, it is differentiable, with
$(g^\delta)'(x)=\frac{v^\delta(x,g^\delta(x))}{u^\delta(x,g^\delta(x))}$.
Thus, by the mean
value theorem,
$$
|(g^\delta)'(x)|\le C'|U^\delta(x,g^\delta(x))-\underline{U}^+|\le
C'\big\|U^\delta-\underline{U}^+\big\|_{L^\infty}\le
C'C_2C_0\varepsilon,
$$ where the constant $C'$ depends only on
$\underline{U}^+$.

\subsection {Finiteness of Collisions and Reflections}
To show that the numbers of fronts/discontinuities and
collisions/reflections do not approach infinity in $\{0<x<\tau\}$
for any finite $\tau>0$, the basic idea presented in
\cite{Holden-Risebro2002}
for the Cauchy problem works well,
but we have to consider additional issues such as the reflections
off the free boundary and the fact that the Euler system is not
strictly hyperbolic in the argument. For completeness, we give the
proof below, which closely follows that in
\cite{Holden-Risebro2002}.

\subsubsection{Generation of Fronts and Modified Construction of
Approximate Solutions}

Firstly, we define the notion of \textit{generation} of a front. We
set that each initial front starting at $x=0$ belongs to the first
generation. Take two first-generation fronts of families $d$ and
$h$, respectively, that collide. The resulting fronts of families
$d$ and $h$ belong to the first generation, while all the
remaining fronts resulting from the collision are called
second-generation fronts. Generally, if a front of family $d$ and
generation $m$ interacts with a front of family $h$ and generation
$n$, the resulting front of families $d$ and $h$ are still of
generation $m$ and $n$, respectively, while the remaining fronts
resulting from this collision are given generation $n+m$. The fronts
of 4-family resulting from reflection of a front $\alpha$ of
1-family off the free boundary  has the same generation of the front
$\alpha$. The point of this notion  is that the fronts of high
generation are quite weak.

Given the approximation parameter $\delta>0$, we remove all fronts
with generation higher than $N$, with
\begin{eqnarray}\label{cutoff}
N=\big[\ln_{4KT}(\delta)\big]
\end{eqnarray}
in our construction of approximate solution $(g^\delta, U^\delta)$.
Here $[z]$ denotes the integer larger than but closest to $z$
and, following the notations in \cite[p.218]{Holden-Risebro2002}, we
set
$$
T=T(x)=\sum_{\alpha\in\Upsilon_x}|\alpha|\le V(x), \qquad
K=\frac{1}{4C_0\varepsilon_0},
$$
with $\varepsilon_0$ sufficiently small
and fixed, and taking later $\varepsilon<\varepsilon_0$,
so that $T< \frac{1}{4K}$.

More precisely, if two fronts of generation $n$ and $m$ collide, at
most two waves will retain their generation. If $n + m > N$, then
the remaining waves will be removed; however, if $n+m \le N$, we use the
original (approximate) solution. When we remove the fronts, we let the
function $U^{\delta}$ be equal to the value that has to be the lower
of the removed fronts, provided that the removed fronts are not the
upmost fronts in the solution of the Riemann problem. If
the upmost are removed, then $U^{\delta}$ is set equal to
the value immediately to the upper of the removed fronts.

We remark that this process of removing (very) weak waves in
approximate Riemann solver in our construction of approximate
solutions will not influence the uniform estimates we obtained in \S 3.2--3.3.
In particular, we still have $T<\frac{1}{4K}$.

\subsubsection{Finiteness of Fronts and Collisions}\label{sec3.4.2}
We will show that there exists only a finite number of fronts of
generation less than or equal to $N$ and that, for a fixed $\delta$,
there is only a finite number of collisions/reflections.

For this, as we know that $T<\frac{1}{4K}$, then the strength of each
individual front is bounded by $\frac{1}{4K}$.
For later reference, we also
note that, by \eqref{cutoff},
\begin{eqnarray}\label{eqN}
(4KT)^{N+1}\le \delta.
\end{eqnarray}

First we consider the number of fronts of first generation. This
number can increase when the first-generation rarefaction fronts
split into several rarefaction fronts. By the term
$\textit{rarefaction front}$ we mean a front approximating a
rarefaction wave. Note that, by the construction of the approximate
Riemann problem, the strength of each split rarefaction front is at
least $\frac{3}{4}\delta$. Given that $T$ is uniformly bounded, we find
\begin{eqnarray}\label{firstgen}
{\text(\rm Number\ of\ first-generation\ fronts)} \le {\text(\rm
Number\ of\ initial\ fronts)} + \frac{4T}{3\delta}.
\end{eqnarray}

Thus, the number of first-generation fronts is finite. This also
means that there will be only a finite number of
collisions/reflections between first-generation fronts and free
boundary. To see this, note first having the assumption of strict
hyperbolicity would have implied that each wave family will have
speeds that are distinct. However, we see that, although the Euler
system is not strictly hyperbolic, the multiplicity of the
eigenvalues is constant for the states $U$ near the background state
$\underline{U}^+$. That is, $\lambda_1(U) < \lambda_2(U) =
\lambda_3(U) < \lambda_4(U)$ for any state $U \in
O_{\epsilon}(\underline{U}^+)$, and hence the eigenvalues are separable
in the same way for any state $U$.

Hence, we can still conclude that
each first-generation front will remain in a wedge in the $(x,y)$--plane
determined by the slowest and fastest speeds of that family.
Eventually, all first-generation fronts will have interacted at most
finite times, and we can also conclude that there can be only a
finite number of collisions between first-generation fronts and free
boundary globally, since once a front is reflected, it will never
meet the free boundary again.

Assuming now that, for some $m\ge1$, there will be only a finite
number of fronts of generation $i$, for all $i < m$, and that there
will only be a finite number of interactions between the fronts and
fronts reflection off free boundary of generation less than $m$.
Then, in analogy to \eqref{firstgen}, we find
\begin{eqnarray}
&&{\text{\rm Number of $m$-th generation fronts}}\nonumber\\
&&\le 2\times({\text{\rm Number of ${j}$-th and ${i}$-th-generation
fronts;}}\  i+j=m) + \frac{4T}{3\delta}<\infty.
\end{eqnarray}
Consequently, the number of fronts of generation less than or equal
to $m$ is finite. We can now repeat the arguments above showing
that there is only a finite number of collisions between
the first-generation fronts (and reflections off free boundary), just
replacing ``first generation'' by ``of generation less than or equal
to $m$'' and show that there is only a finite number of collisions
producing the fronts of generation of $m+1$. Thus, we can conclude that
there is only a finite number of fronts of generation less than
$N+1$, and that these interact (reflect off free boundary) only a
finite number of times.

\section{Convergence and Existence of Weak Entropy Solutions}

In this section we show the strong convergence of a subsequence
of the approximate solutions
to a weak entropy solution of problem \eqref{prob}.

\subsection{Compactness}
We first show there exists a subsequence of approximate solutions
$\{(g^\delta, U^\delta)\}_{\delta>0}$ that
converges to some $(g,U)$ almost everywhere. In \S 4.2,
we show that $(g,U)$ is actually a weak entropy solution
to problem \eqref{prob}.

\subsubsection{Compactness of $\{g^\delta\}$}

We first show the compactness of the approximate free
boundary $\{g^\delta\}_{\delta>0}$. More explicitly, we have

\begin{lemma}\label{compactfb}
Let $g^{\delta}(x)$ be the free boundary for the approximate
solution $U^{\delta}(x)$. Then there is a subsequence $\delta_j\to
0$ so that  $g^{\delta_j}(x) \to g(x)$ uniformly in any compact set.
Furthermore, the limit $g(x)$ is Lipschitz continuous:
$|g(x_1)-g(x_2)|\le C_3\varepsilon|x_1-x_2|$
for some constant $C_3$.
\end{lemma}

\begin{proof} By \eqref{gbound}, that is,
$\|(g^\delta)'\|_{L^\infty([0,\infty))} \le C_3\varepsilon$ and
$g^\delta(0)=0$, we see that, for fixed $T>0$, the family $\{g^\delta\}$ is
uniformly bounded and equicontinuous on $[0,T]$. Then, by the Arzela-Ascoli
compactness criterion, there is a subsequence $\delta_j\to 0$ so
that $g^{\delta_j}\rightrightarrows g$ uniformly for some $g$ in
$[0,T]$ and one easily proves that $|g(x_1)-g(x_2)|\le
C_3\varepsilon|x_1-x_2|$ for $x_1,x_2\in[0,T]$. By taking a diagonal
subsequence for $2T,3T,\ldots$, we can prove that $g$ is defined for
$x\in[0,\infty)$ and $g^{\delta_j}\to g$ uniformly in any compact
subset of $[0,\infty)$, and $|g(x_1)-g(x_2)|\le
C_3\varepsilon|x_1-x_2|$ for any finite $x_1$ and $x_2$.
\end{proof}

\subsubsection{Compactness of $\{U^\delta\}$}
We use the following compactness lemma, which is a modification of
Theorem A.8 in \cite{Holden-Risebro2002}.

\begin{lemma}\label{compactness}
Let $\{u_\eta: [0,\infty) \times [0,\infty)\to\mathbb{R}^4\}_{\eta}$
be a family of functions such that, for each positive $T$,
\begin{itemize}
\item[(a)] $|u_\eta(x,\theta)|\le C_T$ for $(x,\theta)\in[0, T]\times[0,\infty)$
with a constant $C_T$ independent of $\eta$;

\item[(b)] For all  $t\in[0,T]$,
there holds
$$
\sup_{|\xi|\le\rho}\int_B|u_\eta(x, \theta+\xi)-u_\eta(x,\theta)|\,\dd \theta\le
\nu_{B,T}(|\rho|),
$$
for a modulus of continuity $\nu$ and all
compact $B \subset [0,\infty)$ (here $u_\eta(x,t)$ is extended to be
zero for $x \notin [0,\infty)$);

\item[(c)] Furthermore, for any $R>0$, for $s$ and $t$ in $[0,T]$, there holds
$$
\int_0^R|u_\eta(t,\theta)-u_\eta(s, \theta)|\,\dd\theta
\le\omega_{T}(|t-s|)\qquad\text{as}\,\,\, \eta\to0,
$$
for some modulus of continuity $\omega_{T}$.
\end{itemize}
Then there exists a sequence $\eta_j\to 0$ such that, for each
$x\in[0,T]$, the function $u_{\eta_j}(x)$ converges to a function
$u(x)$ in $L^1([0,\infty))$. The convergence is in the topology of
$C([0,T];L^1[0,\infty))$.
\end{lemma}

For any $T>0$, note that $U^\delta(x,y)$ is defined for $0<x<T$ and
$g^\delta(x)<y<\infty$. By introducing $\theta=y-g^\delta(x)$, we
may regard $U^\delta$ as a function of $\theta\in[0,\infty)$  and
$x\in[0,T]$ by defining
$$
\breve{U}^\delta(x,\theta)=U^\delta(x,
\theta + g^{\delta}(x))
$$
to apply Lemma \ref{compactness}.
Obviously
$\norm{\breve{U}^\delta}_{L^\infty}=\norm{U^\delta}_{L^\infty}$, and
$\TV(\breve{U}^\delta)(x,\cdot)=\TV({U}^\delta)(x,\cdot)$.
Then, by
\eqref{linfty}, we see immediately that (a) is valid for
$\{\breve{U}^\delta\}_{\delta>0}$.

Using the boundedness of $L^\infty$ norm and total variation of
$\breve{U}^\delta$ (cf. \eqref{tvestimate}), the verification of (b)
is elementary. Without loss of generality, we assume $\xi>0$. Then
by monotone convergence theorem,
\begin{eqnarray*}
&&\int_{\mathbb{R}^+}|\breve{U}^\delta(x,\theta+\xi)-\breve{U}^\delta(x,\theta)|\,\dd
\theta=\sum_{k=0}^{\infty}\int_{k\xi}^{(k+1)\xi}|\breve{U}^\delta(x,\theta+\xi)
-\breve{U}^\delta(x,\theta)|\,\dd
\theta\\
&&=\int_0^\xi\sum_{k=0}^{\infty}
\big|\breve{U}^\delta(x,z+(k+1)\xi)-\breve{U}^\delta(x,z+k\xi)\big|\,\dd z\\
&&\quad \le (\TV \breve{U}^\delta(x,\cdot))|\xi|\le
(C_2C_0\varepsilon)|\xi|.
\end{eqnarray*}

The verification of (c) is also not difficult. For $0<s<t<T$, we will
prove that
\begin{eqnarray}\label{vc}
\int_{0}^R|\breve{U}^\delta(t,\theta)-\breve{U}^\delta(s,\theta)|\,\dd
\theta\le C(t-s),
\end{eqnarray}
for any $R>0$ and a constant $C$ independent of $\delta, t$, and $s$.

To this end, for given approximate solution $U^\delta$, suppose the
``collision times" are
$$
0<\tau_1<\ldots<\tau_k<\ldots.
$$
Then, for $x\in(\tau_i,\tau_{i+1})$, nothing happens on the (approximate)
free boundary, and then we may ignore the free boundary and write $U^\delta(x,y)$ in
the form
\begin{eqnarray}
U^\delta(x,y)=\sum_{k=1}^{N_i}(U^i_{k+1}-U^i_k)H(y-y^i_k(x))+U^i_{1},
\end{eqnarray}
with $H(\cdot)$ the Heaviside step function (whose value is $0$ for
the negative argument and is $1$ for the positive argument). Here we have
assumed that, for $x\in (\tau_i,\tau_{i+1})$, there are $N_i$
discontinuities  with equation $y=x^i_k(x)$ (from the lower to upper as
$k=1,\ldots, N_i$), and the state in the lower side of $\{y=x^i_k\}$
is $U^i_k$. From \S \ref{sec3.4.2}, we know that $N_i<\infty.$

With the above expression, for $\tau_i< s<t< \tau_{i+1}$, we have
\begin{eqnarray}
&&\int_{\mathbb{R}^+}\big|\breve{U}^\delta(t,\theta)-\breve{U}^\delta(s,\theta)\big|
\,\dd \theta=\int_{\mathbb{R}^+}\left|\int_s^t\frac{\dd}{\dd \tau}
\breve{U}^\delta(\tau,\theta)\,\dd\tau\right|\,\dd \theta\nonumber\\
&&\le \int_{{\mathbb{R}^+}}\int_s^t\sum_{k=1}^{N_i}\big|U^i_{k+1}-U^i_k\big|
\big|H'((g^\delta(\tau) +\theta)-y^i_k(\tau))\big| \left(\left| \frac{\dd
g^\delta(\tau)}{\dd \tau}\right|+ \left|\frac{\dd y^i_k(\tau)}
{\dd \tau}\right|\right)\,\dd\tau\,\dd \theta\nonumber\\
&&\le (L+C_3\varepsilon)
\int_s^t\sum_{k=1}^{N_i}\big|U^i_{k+1}-U^i_k\big|\int_{\mathbb{R}^+}\big|H'((g^\delta(\tau)
+\theta)-y^i_k(\tau))\big|\,\dd
\theta\,\dd\tau\nonumber\\
&&= (L+C_3\varepsilon)
\int_s^t\sum_{k=1}^{N_i}\big|U^i_{k+1}-U^i_k\big|\,\dd\tau\nonumber\\
&&\quad \le (L+C_3\varepsilon)\TV\big(U^\delta(\tau_i+,\cdot)\big)(t-s)\le
\big(L+C_3\varepsilon\big)C_2C_0\varepsilon(t-s).\label{vcc}
\end{eqnarray}
Here we have set
\begin{eqnarray}\label{chspeed}
L=\sup_{U\in
O_\epsilon(\underline{U}^+)}({|\lambda_1(U)|,|\lambda_{2,3}(U)|,
|\lambda_{4}(U)|})
\end{eqnarray}
to be the maximal characteristic
speed, and used the fact that $\left| \frac{\dd y^i_k(x)}{\dd x}\right|\le L$.
Estimate \eqref{gbound} is also used to
control $\big|\frac{\dd g^{\delta}}{\dd \tau}\big|$.

We note \eqref{vcc} also holds for $s=\tau_i$ and/or $t=\tau_{i+1}$.
Then, for $s\in(\tau_i,\tau_{i+1})$ and $t\in (\tau_j,\tau_{j+1})$
with $i<j$, using \eqref{vcc} repeatedly in the intervals
$(s,\tau_{i+1}), (\tau_{i+1},\tau_{i+2}), \ldots,
(\tau_{j-1},\tau_j)$, and $(\tau_j,t)$, we obtain \eqref{vc} with
$C=(L+C_3\varepsilon)C_2C_0\varepsilon$.

Therefore, by Lemma \ref{compactness}, we can find a subsequence
$\{\breve{U}^{\delta_j}\}$ that converges to some $\breve{U}$ under
the metric of $C([0,T]; L^1([0,\infty)))$.
In addition, upon at most a further subsequence, $g^{\delta_j}\to g$. Now set
$U(x,y)=\breve{U}(x,y-g(x))$, which is defined in the domain
$\Omega=\{x>0, y>g(x)\}$, with $D=\{y=g(x)\}$ being the lateral
(free) boundary. In \S 4.2, we show that
$(g,U)$ is actually a weak entropy solution of problem
\eqref{prob}. In the following, for simplification, we also write
$\delta_j$ as $\delta$.

\subsection{Existence of a Weak Entropy Solution}
For $0\le s\le t\le T_0$, define
$\Omega_{s,t}:=\Omega\cap\{x\in[s,t]\}$,
$\Sigma_s=\Omega\cap\{x=s\}$, and $\Gamma_{s,t}=D\cap\{s\le x\le t\}$.
By the
definition of weak entropy solutions (Definition \ref{def11}), a
pair of bounded measurable functions $(g, U)=(g(x), U(x, y))$ is a weak
entropy solution of problem \eqref{prob} provided that
\begin{itemize}
\item  For any $\psi\in C_0^\infty(\mathbb{R}^2)$,
\begin{eqnarray}
F_s^t(U):=\int_{\Omega_{s,t}} \big(\rho u\p_x\psi+\rho
v\p_y\psi\big)\,\dd y\,\dd x+\int_{\Sigma_s}\rho u\psi\,\dd y
-\int_{\Sigma_t}\rho u\psi\,\dd y=0;
\end{eqnarray}

\item For any $\psi\in C_0^\infty(\mathbb{R}^2)$,
\begin{eqnarray}
G_s^t(U)&:=&\int_{\Omega_{s,t}} \big((\rho u^2+p)\p_x\psi+
\rho uv\p_y\psi\big)\,\dd y\,\dd x+\int_{\Sigma_s}(\rho
u^2+p)\psi\,\dd y \nonumber\\
&&-\int_{\Sigma_t} (\rho u^2+p)\psi\,\dd
y-\underline{p}\int_{\Gamma_{s,t}}\psi n_1\,\dd s=0;
\end{eqnarray}

\item For any $\psi\in C_0^\infty(\mathbb{R}^2)$,
\begin{eqnarray}
I_s^t(U)&:=&\int_{\Omega_{s,t}} \big(\rho uv\p_x\psi+
(\rho v^2 + p)\p_y\psi\big)\,\dd y\,\dd x
+\int_{\Sigma_s}\rho uv\psi\,\dd y
-\int_{\Sigma_t} \rho uv\psi\,\dd y\nonumber\\&&
-\underline{p}\int_{\Gamma_{s,t}}\psi n_2\,\dd s=0;
\end{eqnarray}

\item For any $\psi\in C_0^\infty(\mathbb{R}^2)$,
\begin{eqnarray}
J_s^t(U)&:=&\int_{\Omega_{s,t}} \big(\rho u(E+\frac{p}{\rho})\p_x\psi+
\rho v(E+\frac{p}{\rho})\p_y\psi\big)\,\dd y\,\dd x\nonumber\\
&&+\int_{\Sigma_s} \rho u(E+\frac{p}{\rho})\psi\,\dd y
-\int_{\Sigma_t} \rho u(E+\frac{p}{\rho})\psi\,\dd y=0;
\end{eqnarray}

\item For any $\psi\in C_0^\infty(\mathbb{R}^2)$ that is nonnegative,
\begin{eqnarray}
E_s^t(U):=\int_{\Omega_{s,t}}\big( \rho uS\p_x\psi
+\rho vS\p_y\psi)\,\dd y\,\dd x
+\int_{\Sigma_s}\rho uS\psi\,\dd y
-\int_{\Sigma_t} \rho uS\psi\,\dd y\le0.
\end{eqnarray}
\end{itemize}

\subsubsection{Estimate on the Total Strength of the Removed Fronts}
For any approximate solution $(g^\delta,U^\delta)$, we set
$$
\Omega^\delta:=\{x>0, y>g^ \delta\}
$$
and
$$
\Gamma^\delta:=\{y=g^\delta(x)\}.
$$
For $0\le s\le t\le T_0$, define
$$
\Omega^\delta_{s,t}:=\Omega^\delta\cap\{x\in[s,t]\}, \qquad
\Sigma^\delta_s=\Omega^\delta\cap\{x=s\}, \qquad
\Gamma^\delta_{s,t}=\Gamma^\delta\cap\{x\in[s,t]\}.
$$
We note by our
construction of approximate solutions that $U^\delta$ may not be a
weak entropy solution of the Euler equations \eqref{euler} in
$\Omega^\delta$, since there are possible errors introduced by
the approximating rarefaction wave via several fronts,
and removing weak
fronts of higher generations.
In the following we will estimate
these errors and show that they actually vanish as $\delta\to0$.
The analysis is again quite similar to
\cite{Holden-Risebro2002}.
 We first list below Lemma
 6.5 in \cite{Holden-Risebro2002} for later reference.

\begin{lemma}\label{lem42}
Let $\mathcal{G}_m$ denote the set of all fronts of generation $m$,
and let $\mathcal{T}_m$ denote the sum of the strengths of fronts of
generation $m$:
$\mathcal{T}_m=\sum_{\alpha_j\in\mathcal{G}_m}|\alpha_j|$. Then
$T=\sum_{m=1}^N\mathcal{T}_m$, and
$$\mathcal{T}_m\le C(4KT)^m$$
for some constant $C$. In particular, for $m=N+1$, we have
$\mathcal{T}_{N+1}\le C\delta$ (cf. \eqref{eqN}).
\end{lemma}

\subsubsection{Exact Riemann Solutions}
For a given approximate solution $(g^\delta, U^\delta)$, suppose as
before that the collision/reflection ``times" are
$x=\tau_1<\tau_2<\ldots$. For a fixed interval
$[\tau_j,\tau_{j+1}]$, set $s_1=\tau_j$. We solve the following
initial--free boundary problem with $i=1$ (cf. \eqref{eq21}):
\begin{eqnarray}\label{probconver}
\begin{cases}\p_xW(\tilde{U})+\p_yH(\tilde{U})=0, &x>s_i,\ \
y>\tilde{g}(x),\\
\tilde{U}=U^\delta,& x=s_i,\ \
y>\tilde{g}(s_i):=g^\delta(s_i),\\
\tilde{p}=\underline{p}, & x>s_i,\ \ \text{on }\ \tilde{\Gamma}:=\{y=\tilde{g}(x)\},\\
\tilde{v}=\tilde{g}^{\prime}\tilde{u}, & x>s_i,\ \text{on }\ \
\tilde{\Gamma}.\end{cases}
\end{eqnarray}
Since the ``initial data" $U^\delta(s_1,\cdot)$ is piecewise
constant, the solution $(\tilde{g}_1, \tilde{U}_1)$ is obtained by
solving the Riemann problems. It can be solved up to $x=s_2$ when
two waves interaction or reflection off the free boundary occurs (if
$s_2>\tau_{j+1}$, we set $s_2=\tau_{j+1}$). Then we solve
$(\tilde{g}_2, \tilde{U}_2)$ from problem \eqref{probconver} with
$i=2$ (note that the initial data is $U^\delta(s_2,\cdot)$), up to some
$s_3$. Repeat this process, we obtain
$$
(\tilde{g}_i,\tilde{U}_i)\qquad \mbox{in $[s_i,s_{i+1})$}
$$
with
$\cup_{i=1}^\infty[s_i,s_{i+1})=[\tau_j,\tau_{j+1})$. We can then
define  $(\tilde{g},\tilde{U})$ piecewise in
$x\in[\tau_j,\tau_{j+1})$ by
$(\tilde{g},\tilde{U})=(\tilde{g}_i,\tilde{U}_i)$ for
$x\in[s_i,s_{i+1})$.

\subsubsection{Error of Splitting Rarefaction Waves}
Let $\tilde{U}^\delta$ be the {\it approximate} solution obtained
from problem \eqref{probconver} in $[s_{i},s_{i+1}]$, with the
approximating parameter $\delta$.
This means that the rarefaction waves
in $\tilde{U}$ are separated into many discontinuities; while there
is no front to be removed since each front in $\tilde{U}$ is of
generation one. Also, by our rule of splitting rarefaction waves,
the lowermost state of $\tilde{U}^\delta$ is the same as
$\tilde{U}$. This implies that the corresponding free boundaries are the
same, and both $\tilde{U},\ \tilde{U}^\delta$ are defined in the same
domain. The analysis below is similar to
\cite{Holden-Risebro2002}.
We present details here to
show the ideas there still work for our free boundary problem.

Suppose that there is a rarefaction wave in $\tilde{U}$ with the lower
state $\tilde{U}_l$ and upper state $\tilde{U}_r$. Then this
rarefaction wave is replaced by a step function $\tilde{U}^\delta$.
There also holds
$$
|\tilde{U}^\delta(x,y)-\tilde{U}(x,y)|\le
O(\delta)
$$
by our splitting process (it is zero for the points not in
rarefaction wave fan). We also want to find the error in the $L^1$--space.
To
this end, we note that there are at most
$
\frac{|\tilde{U}_r-\tilde{U}_l|}{O(\delta)}
$
steps, and the width of each
step is at most $(x-s_i)\triangle\lambda$, with $\triangle\lambda$
the difference of characteristic speeds of two adjacent approximate
fronts of each step
--- it is less than $O(\delta)$ (cf. \eqref{eqappr}). Using the mean
value theorem (since we know uniform $L^\infty$ bounds of
$\tilde{U}$ and $\tilde{U}^\delta$), and summing up for all
rarefaction wave fans across $x$, we find
\begin{eqnarray}\label{411}
&&\int_{y>\tilde{g}(x)}|W(\tilde{U}^\delta)(x,y)-W(\tilde{U})(x,y)|\,\dd
y\le
C\int_{y>\tilde{g}(x)}|\tilde{U}^\delta(x,y)-\tilde{U}(x,y)|\,\dd
y\nonumber\\
&&\le O(\delta)\sum_{k} |\tilde{U}_r^k-\tilde{U}_l^k|
|x-s_i|
\le  O(\delta) {\rm{T.V.}}(U^\delta)|x-s_i|=O(\delta)|x-s_i|.
\end{eqnarray}
We note here that $\sum_{k}{|\tilde{U}_r^k-\tilde{U}_l^k|}$ is
actually controlled by the total variation of the initial data by using the property
of the Riemann solution. Similar inequality also holds when $W(U)$ is
replaced by $H(U)$.

\subsubsection{Error of the Removing Weak Fronts}
We then compare $\tilde{U}^\delta$ and $U^\delta$ in
$x\in[s_i,s_{i+1}].$ We note that both $\tilde{U}^\delta$ and
$U^\delta$ satisfy the same initial data. The only difference
between them is that some fronts in $\tilde{U}^\delta$ of generation
$N+1$ are ignored to obtain $U^\delta$. Note that, by the removing fronts
of generation $N+1$, we always keep the lowermost state
the same as before. This means that the free boundary of $U^\delta$ is
the same as $\tilde{U}^\delta$, hence still to be
$y=\tilde{g}(x)=g^\delta(x)$. Consequently, $\tilde{U}^\delta$ is
different from $U^\delta$ in $x\in(s_i,s_{i+1})$ only in a number of
wedges emanating from the discontinuities in $U^\delta(s_i,\cdot)$,
and in each wedge, the difference is bounded by the strength of the
removing fronts $\alpha$ that are of generation $N+1$. We also note
the width of each wedge is controlled by $O(x-s_i)$.  By Lemma
\ref{lem42}, we then find
\begin{eqnarray}\label{412}
&&\int_{y>\tilde{g}(x)}\big|W(\tilde{U}^\delta)(x,y)-W({U}^\delta)(x,y)\big|\,\dd
y\le
C\int_{y>\tilde{g}(x)}\big|\tilde{U}^\delta(x,y)-{U}^\delta(x,y)\big|\,\dd
y\nonumber\\
&&\le O(|x-s_i|) \sum_{\alpha\in\mathcal{G}_{N+1}}|\alpha|
 \le O(\delta)|x-s_i| .
\end{eqnarray}
Similar inequality is also true for $H(U)$.

\subsubsection{Total Error of Approximate Solutions}

Since $\tilde{U}$ is obtained by the exact Riemann solvers for
$x\in[s_i,s_{i+1}]$, there must hold (with $\Omega_{s,t}$ and
$\Sigma_s$, and $\Gamma_{s,t}$ in the integrals replaced by
$\Omega^\delta_{s,t}$, $\Sigma^\delta_s$, and $\Gamma_{s,t}^\delta$
respectively, since we have shown that the free boundary of
$\tilde{U}$ is the same as $U^\delta$):
$$
F_{s_i}^{s_{i+1}}(\tilde{U})=0,\quad G_{s_i}^{s_{i+1}}(\tilde{U})=0,
\quad I_{s_i}^{s_{i+1}}(\tilde{U})=0, \quad
J_{s_i}^{s_{i+1}}(\tilde{U})=0, \quad
E_{s_i}^{s_{i+1}}(\tilde{U})\le0.
$$
From \eqref{411} and
\eqref{412}, we also obtain that, for any $x\in[s_i,s_{i+1}]$,
\begin{eqnarray}
&&\int_{\Sigma^\delta_x}\big|W({U}^\delta)(x,y)-W(\tilde{U})(x,y)\big|\,\dd y\le O(\delta) |x-s_i|,\\
&&\int_{\Sigma^\delta_x}\big|H({U}^\delta)(x,y)-H(\tilde{U})(x,y)\big|\,\dd
y\le O(\delta)|x-s_i|.
\end{eqnarray}
Therefore, as an example, we find (note that the boundary term involving the
pressure $\underline{p}\int_{\Gamma}\psi n_1\,\dd s$ canceled because the
boundary is the same):
\begin{eqnarray*}
\big|G_{s_i}^{s_{i+1}}(U^\delta)\big|
&=&\big|G_{s_i}^{s_{i+1}}(U^\delta)-G_{s_i}^{s_{i+1}}(\tilde{U})\big|
\\&\le&
\int_{s_i}^{s_{i+1}}\int_{\Sigma^\delta_x}\big|\big(W_2(U^\delta)-W_2(\tilde{U})\big)\p_x\phi
+\big(H_2(U^\delta)-H_2(\tilde{U})\big)\p_y\phi\big|\,\dd y\dd
x\\&&+\int_{\Sigma^\delta_{s_i}}\big|\big(W_2(U^\delta)-W_2(\tilde{U})\big)\phi\big|\,\dd
y+\int_{\Sigma^\delta_{s_{i+1}}}\big|(W_2(U^\delta)-W_2(\tilde{U}))\phi\big|\,\dd y\\
&\le&
M\int_{s_i}^{s_{i+1}}\int_{\Sigma^\delta_x}\big|W_2(U^\delta)-W_2(\tilde{U})\big|\,\dd
y\,\dd
x+M\int_{s_i}^{s_{i+1}}\int_{\Sigma^\delta_x}\big|H_2(U^\delta)-H_2(\tilde{U})\big|\,\dd
y\,\dd
x\\
&&+M\int_{\Sigma^\delta_{s_i}}\big|W_2(U^\delta)-W_2(\tilde{U})\big|\,\dd
y+M\int_{\Sigma^\delta_{s_{i+1}}}\big|W_2(U^\delta)-W_2(\tilde{U})\big|\,\dd y\\
&\le& {O}(\delta)(s_{i+1}-s_i)^2+{O}(\delta)(s_{i+1}-s_i),
\end{eqnarray*}
where $M:=\norm{\phi}_{W^{1,\infty}}$. Then we find
$$
\big|G_{\tau_j}^{\tau_{j+1}}(U^\delta)\big|
\le\mathcal{O}(\delta)\sum_{i=1}^\infty\big((
s_{i+1}-s_i)^2+(s_{i+1}-s_i)\big)\le
O(\delta)\big((\tau_{j+1}-\tau_j)^2+(\tau_{j+1}-\tau_j)\big).
$$
Thus it is
clear that
$$
\big|G_{s}^{t}(U^\delta)\big|\le O(\delta)(|t-s|^2+|t-s|) \qquad\mbox{for
any $0\le s<t\le \infty$},
$$
and
$$
\lim_{\delta\to 0}G_{s}^{t}(U^\delta)=0.
$$
We now need to prove
$$
\lim_{\delta\to 0}G_{s}^{t}(U^\delta)=G_{s}^{t}(U).
$$

\subsubsection{Verification of Weak Entropy Solutions}
Set
$$
\breve{\phi}^\delta(x,\theta)=\phi(x, \theta+g^{\delta}(x)),\quad
\breve{W}^{\delta}_2(x,\theta)=
W_2(U^{\delta})(x,\theta+g^\delta(x)),\quad
\breve{H}^{\delta}_2(x,\theta)=
H_2(U^{\delta})(x,\theta+g^\delta(x)),
$$
where $W^{\delta}_2=
W_2(U^{\delta})$ and $H^{\delta}_2= H_2(U^{\delta})$.
Then we have
\begin{eqnarray*}
G_{s}^{t}(U^\delta)&=&\int_s^t\int_{y>g^\delta(x)}
\big(W_2^{\delta}\p_x\phi + H_2^{\delta}\p_y\phi\big) \,\dd y\dd x\\
&&+\int_{y>g^\delta(s)}\big(W_2^{\delta}\phi\big)|_{x=s}\,\dd
y-\int_{y>g^\delta(t)}
\big(W_2^{\delta}\phi\big)|_{x=t}\,\dd y-\underline{p}\int_{\Gamma^\delta_{s,t}}
\phi n_1\,\dd s \\
&=&\int_s^t\int_{\mathbb{R}^+}\big(\breve{W}_2^{\delta}(\p_x\breve{\phi}^\delta-
\p_\theta\breve{\phi}^\delta (g^\delta)')+\breve{H}_2^{\delta}\p_\theta
\breve{\phi}^\delta\big)\,\dd \theta\dd x\\
&&+\int_{\mathbb{R}^+}
\big(\breve{W}_2^{\delta}\breve{\phi}^\delta\big)|_{x=s}\,\dd\theta-
\int_{\mathbb{R}^+}
\big(\breve{W}_2^{\delta}\breve{\phi}^\delta\big)|_{x=t}\,\dd\theta\\
&&-\underline{p}\int_s^t\phi(x,g^\delta(x))\frac{(g^\delta)'(x)}{\sqrt{1+((g^\delta)'(x))^2}}\sqrt{1+((g^\delta)'(x))^2}\,\dd
x.
\end{eqnarray*}

By Lemma \ref{compactfb}, we know that $g^\delta\to g$ uniformly for
$x\in[s,t]$. Since $\breve{U}^\delta$ is uniformly bounded and
converges to $\breve{U}$ under the metric of
$C([s,t];L^1(\mathbb{R}^+))$, then $\breve{W}^\delta$ and
$\breve{H}^\delta$ are also uniformly bounded and converges to
$\breve{W}$ and $\breve{H}$ respectively in the topology of
$C([s,t];L^1(\mathbb{R}^+))$.
From these facts, one can easily use
the Lebesgue dominant convergence theorem to show (with
$\breve{\phi}=\phi(x,\theta+g(x))$) that, as $\delta\to 0$,
\begin{eqnarray*}
&&\int_{\mathbb{R}^+}
\big(\breve{W}_2^{\delta}\breve{\phi}^\delta\big)|_{x=s}\,\dd\theta-
\int_{\mathbb{R}^+}
\big(\breve{W}_2^{\delta}\breve{\phi}^\delta\big)|_{x=t}\,\dd\theta\to\int_{\mathbb{R}^+}
\big(\breve{W}_2\breve{\phi}\}|_{x=s}\,\dd\theta-\int_{\mathbb{R}^+}
\big(\breve{W}_2\breve{\phi}\big)|_{x=t}\,\dd\theta,\\
&&\int_s^t\int_{\mathbb{R}^+}\big(\breve{W}_2^{\delta}\p_x\breve{\phi}^\delta+\breve{H}_2^{\delta}\p_\theta
\breve{\phi}^\delta\big)\,\dd \theta\dd
x\to\int_s^t\int_{\mathbb{R}^+}\big(\breve{W}_2\p_x\breve{\phi}+
\breve{H}_2\p_\theta \breve{\phi}\big)\,\dd \theta\dd x.
\end{eqnarray*}
Since $\{(g^\delta)'\}$ is uniformly bounded, we may assume that
$(g^\delta)'\rightharpoonup h$ in the weak* topology of
$L^\infty(\mathbb{R}^+)$.
Since $g^\delta\to g$ uniformly in
$[s,t]$, we find that $(g^\delta)'\to g'$ in the sense of distributions.
Thus, we must have $h=g'$. Therefore, as $\phi(x,g^\delta(x))\to \phi(x,g(x))$
uniformly in $[s,t]$, and $(g^\delta)'\rightharpoonup   g'$ in the weak* $L^\infty$,
we have
\begin{eqnarray*}
\int_s^t\phi(x,g^\delta)(g^\delta)'(x)\,\dd
x\to\int_s^t\phi(x,g) g'(x)\,\dd x.
\end{eqnarray*}
We also find
\begin{eqnarray*}
&&\int_{s}^t\int_{\mathbb{R}^+}\big(\breve{W}_2^\delta\p_\theta\breve{\phi}^\delta(g^\delta)'
-\breve{W}_2\p_\theta\breve{\phi}g'\big)\,\dd\theta\dd x\\
&&=\int_{s}^t\int_{\mathbb{R}^+}\breve{W}_2^\delta\big(\p_\theta\breve{\phi}
^\delta-\p_\theta\breve{\phi}\big)(g^\delta)'\,\dd\theta\dd x+
\int_{s}^t\int_{\mathbb{R}^+}\breve{W}_2^\delta\p_\theta\breve{\phi}\big(
(g^\delta)'-g'\big)\,\dd\theta\dd x\\
&&\quad +\int_{s}^t\int_{\mathbb{R}^+}(\breve{W}_2^\delta-\breve{W}_2)
\p_\theta\breve{\phi}g'(x) \,\dd\theta\dd x.
\end{eqnarray*}
Using the boundedness of $\{\breve{W}_2^\delta\}$ and $\{(g^\delta)'\}$,
and the uniform convergence $\p_\theta\breve{\phi}^\delta\to
\p_\theta\breve{\phi}$, the first integral in the right-hand side
goes to zero as $\delta\to0$.
The third one converges to zero
follows directly from $\breve{W}_2^\delta\to \breve{W}_2$ in
$C([s,t];L^1(\mathbb{R}^+))$.
For the second integral, it can
be written as
\begin{eqnarray*}
\int_{s}^t\int_{\mathbb{R}^+}\big(\breve{W}_2^\delta-\breve{W}_2\big)\p_\theta\breve{\phi}
\big((g^\delta)'-g'\big)\,\dd\theta\dd
x+\int_{s}^t\int_{\mathbb{R}^+}\breve{W}_2\p_\theta\breve{\phi}\big(
(g^\delta)'-g'\big)\,\dd\theta\dd x.
\end{eqnarray*}
By the boundedness of $(g^\delta)'-g'$, the first one then converges to
zero; for the second one, using again $(g^\delta)'\rightharpoonup
g'$ in the weak* topology of $L^\infty$. Then we have also proved
\begin{eqnarray*}
&&\int_{s}^t\int_{\mathbb{R}^+}\breve{W}_2^\delta\p_\theta\breve{\phi}^\delta(g^\delta)'
\,\dd\theta\dd x \to\int_{s}^t\int_{\mathbb{R}^+}
\breve{W}_2\p_\theta\breve{\phi}g'\,\dd\theta\dd x.
\end{eqnarray*}
Hence, we have
\begin{eqnarray*}
\lim_{\delta\to 0}G_{s}^{t}(U^\delta)&=&\int_s^t\int_{\mathbb{R}^+}
\big(\breve{W}_2(\p_x\breve{\phi}-\p_\theta\breve{\phi}g') + \breve{H}_2
\p_\theta\breve{\phi}\big) \,\dd \theta\dd x\\
&&+\int_{\mathbb{R}^+}
\big(\breve{W}_2\breve{\phi}\big)|_{x=s}\,\dd \theta
-\int_{\mathbb{R}^+}\big(\breve{W}_2\breve{\phi}\big)|_{x=t}\,\dd \theta
-\underline{p}\int_s^t\phi(x,g(x))g'(x)\,\dd x \\
&=&G_s^t(U)
\end{eqnarray*}
by a change of variables $(x,y)=(x,\theta+g(x))$.

Therefore, we have proved $G_s^t(U)=0$ as desired.
Similarly, we can
conclude
$$
E_s^t(U)\le0, \qquad F_s^t(U)=0, \qquad I_s^t(U)=0, \qquad J_s^t(U)=0,
$$
and hence the limit $(g, U)$ obtained from the approximate solutions
$(g^\delta,U^\delta)$ is actually a weak entropy solution to problem
\eqref{prob}.

It is clear that $g$ should satisfy the estimate listed in
Theorem \ref{thm1}, as guaranteed by Lemma \ref{compactfb}.
To show
$\norm{(U-\underline{U}^+)(x,\cdot)}_{\rm BV}\le C\varepsilon$, we
note that we have proved
$$
\big\|(U^\delta-\underline{U}^+)(x,\cdot)\|_{\rm BV}\le
C\varepsilon.
$$
Then, by Helly's theorem, without loss of generality, we
may assume
$$
(U^\delta-\underline{U}^+)(x,\cdot)\to
(\tilde{U}-\underline{U}^+)(x,\cdot) \qquad \mbox{pointwise}
$$
for some
$(\tilde{U}-\underline{U}^+)(x,\cdot)$ so that
$$
\big\|(\tilde{U}-\underline{U}^+)(x,\cdot)\big\|_{\rm BV}\le
C\varepsilon \qquad\mbox{as $\delta\to 0$}.
$$
However, by uniqueness of the pointwise
limit, we must have $\tilde{U}=U$.
This completes the proof.

\section{Asymptotic Behavior of Weak Entropy Solutions}
Finally we discuss the asymptotic behavior of the weak
entropy solution $(g,U)$ as $x\to\infty$.

For any given $\delta>0$ and the corresponding approximate solution
$(g^\delta,U^\delta)$, we
know that there are a finite number of fronts and collisions/reflections.
Thus, there exists $x_\delta>0$ so that, for $x>x_\delta$, there is no
collisions and reflections.
Suppose then that there are $m+1$
different states $\{U^\delta_j\}_{j=0}^m$ from the upper to lower. It is
obvious that $U^\delta_0=\underline{U}_0$ (cf. Remark
\ref{remark1}), and there is $m_0$ with $1\le m_0<m$ so that each
pair $(U^\delta_{j-1}, U^\delta_j)$ ($j=1,\ldots, m_0$) is connected
by a discontinuity of the first characteristic family, while each
$(U^\delta_{j-1}, U^\delta_j)$, $j=m_0+1,\ldots, m$, is connected by
a characteristic discontinuity (of the second and/or third characteristic
family). Since no fronts interact, it is only possible that, for
$m_0>1$, all the discontinuities of the first family must be
rarefaction waves; for $m_0=1$, this discontinuity might be a shock
or a rarefaction wave. For states $U^\delta_{j}$ ($j=m_0+1,\ldots,
m$), the pressure must be $\underline{p}$ by the boundary condition
and the Rankine--Hugoniot jump conditions of characteristic discontinuities.

Now we solve the free boundary Riemann problem of the Euler equations
\eqref{euler} with the initial data $U=\underline{U}_0$ and the boundary
condition $p=\underline{p}$ on the free boundary $y=kx$. Suppose that the
solution is given by $U_\infty=\Psi_4(\alpha;\underline{U}_0)$. Then
$k=\frac{v_\infty}{u_\infty}$, and the resulting 4-wave is a shock if
$\underline{p}>\underline{p}_0$, and a rarefaction wave if
$\underline{p}<\underline{p}_0$. It is clear that both
$\frac{v^\delta_j}{u^\delta_j}$ ($j=m_0+1,\ldots, m$) and
$(g^\delta)'$ should be $\frac{v_\infty}{u_\infty}$ for all $x>x_\delta$.

We now define $p_\delta(\xi,\theta)=p^\delta(\xi+x_\delta+1,
\theta+g^\delta(\xi+x_\delta))$ for $\xi\ge0$ and $\theta\ge0$. It
is easy to see that $|p_\delta(\xi,\cdot)|$  and
${\rm{T.V.}}(p_\delta)(\xi,\cdot)=|\underline{p}_0-\underline{p}|$
are bounded for all $\delta>0$ and  given $\xi$. Thus, by Helly's
theorem, there is a subsequence (still denoted as $\delta$) so that
$\lim_{\delta\to 0}p_\delta(\xi,\theta)=p(\theta)=\underline{p}$ for
$\theta>0$ pointwise. This should imply that, for a.e. $\theta\ge0$
and the weak entropy solution $U=(u,v,p,\rho)$,
\begin{eqnarray*}
\lim_{x\to\infty}p(x,\theta+g(x))=\underline{p}.
\end{eqnarray*}
Similarly, we have
\begin{eqnarray*}
\lim_{x\to\infty}\frac{v}{u}(x,\theta+g(x))=\frac{v_\infty}{u_\infty}.
\end{eqnarray*}
It then follows from $g'=\frac{v}{u}$ that
\begin{eqnarray*}
\lim_{x\to\infty}g'(x)=\frac{v_\infty}{u_\infty}.
\end{eqnarray*}

\medskip
\bigskip

{\bf Acknowledgments.} The research of
Gui-Qiang Chen was supported in part by the National Science
Foundation under Grant DMS-0807551, the UK EPSRC Science and Innovation
award to the Oxford Centre for Nonlinear PDE (EP/E035027/1),
the NSFC under a joint project Grant 10728101, and
the Royal Society--Wolfson Research Merit Award (UK).
Vaibhav Kukreja was supported in part by the National Science
Foundation under Grant DMS-0807551, the UK EPSRC Science and Innovation
award to the Oxford Centre for Nonlinear PDE (EP/E035027/1).
Hairong Yuan is supported in part by National
Natural Science Foundation of China under Grant No. 10901052, China
Scholarship Council (No. 2010831365), Chenguang Program (No. 09CG20)
sponsored by Shanghai Municipal Education Commission and Shanghai
Educational Development Foundation, and a Fundamental Research Funds
for the Central Universities.


\end{document}